\documentclass[12pt]{article}
\usepackage{amsmath,amssymb,amsthm}
\usepackage{graphics,epsfig}
\usepackage{hyperref}
\usepackage{natbib}
\usepackage{color}
\usepackage{graphicx}
\usepackage{caption}
\usepackage{subcaption}
\usepackage{float}
\usepackage{dsfont}
\usepackage{mathrsfs}
\usepackage{multirow}
\usepackage{comment}

\def \ra {\rightarrow}

\def \E {\mathbb{E}}

\def \a {\alpha}
\def \be {\beta}

\newtheorem{definition}{\bf Definition}
\newtheorem{defn}[definition]{\bf Definition}

	\newtheorem{remark}{\bf Remark}
	\newtheorem{rmk}[remark]{\bf Remark}

	\newtheorem{theorem}{\bf Theorem}
	
	\newtheorem{prop}{\bf Proposition}
	\newtheorem{lem}[theorem]{\bf Lemma}
	\newtheorem{as}{\bf Assumption}
	\newtheorem{ax}{\bf Axiom}


\begin{document}

  \title{\bfseries Semicooperation under curved strategy spacetime}

  \author{
  	Paramahansa Pramanik\footnote{e-mails of authors: {\small\texttt{ppramanik1@niu.edu}},\ {\small\texttt{polansky@niu.edu}}}\; \footnote{Department of Mathematical Sciences, Northern Illinois University, DeKalb, IL, 60115,
  		United States.}\; \;  
  	Alan M. Polansky \footnote{Department of Statistics and Actuarial Science, Northern Illinois University, DeKalb, IL, 60115,
  		United States.}
  }

\date{\today}
\maketitle


\subparagraph{Abstract.} Mutually beneficial cooperation is a common part of economic systems as firms in partial cooperation with others can often make a higher sustainable profit. Though cooperative games were popular in $1950$s, recent interest in non-cooperative games is prevalent despite the fact that cooperative bargaining seems to be more useful in economic and political applications. In this paper we assume that the strategy space and time are inseparable with respect to a contract. Under this assumption we show that the strategy spacetime is a dynamic curved Liouville-like $2$-brane quantum gravity surface under asymmetric information and that traditional Euclidean geometry fails to give a proper feedback Nash equilibrium. Cooperation occurs when two firms' strategies fall into each other's influence curvature in this strategy spacetime. Small firms in an economy dominated by large firms are subject to the infuence of large firms. We determine an optimal feedback semicooperation of the small firm in this case using a Liouville-Feynman path integral method.



\section{Introduction}
In this paper we discuss a generalized concept of cooperation by assuming that all firms behave rationally under asymmetric information. Their economic influence creates curvature in strategy spacetime. Furthermore, we assume the information set is incomplete and imperfect. By rationality, a firm uses all currently available information and resources to make its decisions. Under incomplete and imperfect information, two firms do not have any prior knowledge about each other and make guesses based on the available information in economy. Every firm has its own dynamic strategy to do business. Therefore, there exists a polygonal curved strategy space for each firm formed by the strategies taken by it historically. In this paper we include time as an important component of this curved strategy space and we define a curved strategy spacetime. 

\medskip

Time is an important aspect of strategy. Decisions made at one time-point may affect the number and type of strategies at a later time. This implies that the shape of the strategy set is time dependent. Alternatively, if the market environments are stable then a firm might not need to cooperate with a new firm and the strategy may be stable. This behavior depends on the available information to the firm. If information is perfect and complete, then keeping the same strategy is sustainable. However, under asymmetric information a stable strategy is not possible and the shape of the strategy polygon is probabilistic in time. 

\medskip
 
Our main focus on this paper is semi-cooperative games. The main reason for the popularity of cooperative games arises from its power to rule out externalities \cite{maskin2016}. The smallest cooperation does consider the externalities which is the heart of the bargaining power. We also assume that a big existing firm in an economy creates a curvature around itself by showing off more market or political power or doing commercials. Because of the presence of asymmetric information a small firm has inexact knowledge of the big firm. Therefore, an existing big firm's strategy should to create a higher degree of curvature in the strategy spacetime to attract more fringes and attain higher profit. On the other hand, if there is existence of multiple big firms, a new fringe finds out multiple curvatures and breaks apart into small pieces in terms of the share of its value and creates semi-cooperation like firm A's cooperation with firms B,C, and D. If two firms have same market power then, in the strategy spacetime when they come closer, they feel curvature in each other's strategy and form a cooperation. From Maskin we know, in the literature cooperative games particularly in Shapley value \cite{shapley1953}, only the resulting payoffs leads to a coalition instead of individual strategies. In this paper show that individual strategies create curvature in strategy spacetime, and this curvature can lead to a coalition.
 
 \section{Bounded rationality in curved strategy spacetime}

Bounded rationality implies that cognitivity and time limitations make a firm unable to make rational decisions (\cite{gigerenzer2002},\cite{simon1955}). A common assumption in economics for constrained optimization is that a firm has Laplacean superintelligence and has unlimited resources \cite{gigerenzer2002}. If we incorporate asymmetric information then constrained optimization under rationality yields us to a different set of  results. Even under perfect and complete information, an individual might not make rational decisions (\cite{allais2008}, \cite{machina1987}).

\medskip

To motivate the mathematical development that follows, consider  an  incumbent firm A with a large market share that behaves rationally, has a sustainable profit, and has sufficient economic influence such that a curvature in the strategy spacetime exists in the neighborhood of the firm's strategy.  Suppose that a firm Bwith a much smaller market share considers options in the strategy space. These options correspond to a geodesic polygon in the strategy space, with the purpose of finding options to obtain an optimal sustainable profit or to move in the direction of bounded rationality. Firm B decides according to its bounded rationality and finds that collusion with Firm A provides a higher profit potential and Firm B falls in Firm A's implied curvature in the strategy spacetime. Therefore, Firm B cannot distinguish between its bounded rationality and the curvature created by the market power and rationality of Firm A. Firm B finds both of them are equivalent. Furthermore, under the existence of multiple big firms, Firm B may consider coalitions with more than one firm may be more profitable. In this case Firm B is influenced by multiple curvatures created by multiple big firms. We define this type of cooperation as semi-cooperation. Sometimes incumbent firms cooperate to block other firms from entering an economy. Under asymmetric information the new firm cannot see difference between the peak in the strategy spacetime and the reduction of profit caused by internal cooperation of the incumbent firms. Hence, bounded rationality and curvature of strategy space are equivalent to each other. If we add the time dimension, the strategy polygons may change their shapes based on whether the players cooperate or not. As the strategy of Firm B enters the region where Firm A has influence it must either cooperate with Firm A or leave the economy.

\medskip

Suppose that $\mathcal{M}$ is an $n$-dimensional manifold such that it is Hausdorff, locally Euclidean of dimension $n$, and has a countable basis of open strategy sets from where a firm can choose a strategy polygon \cite{boothby1986}. Throughout this paper we assume $n=3$ where the two horizontal axes represent two real dimensions of the strategy spacetime and the vertical dimension is imaginary time, making $\mathcal{M}$ complex with the signature $(-,+,+)$. On this complex manifold we create a smooth $C^\infty$ structure  and we allow firms to come and interact with each other \cite{boothby1986}. 

\medskip

Each firm has a measure of market share $x(s)\in D$ at time $s\in[0,t]$ which creates a curvature in the strategy spacetime where $D$ is the domain in this complex manifold $\mathcal{M}$. Under non-cooperation the strategy cannot go beyond the strategy-set. As the firm creates a curve, the strategy set becomes a complex geodesic polygon. Following  \cite{danielsen1989} and \cite{sjoberg2006} we can write in $\mathcal{M}$ the geodesic line $\overrightarrow{PQ}$ connecting two points $P$ and $Q$ with latitudes, longitudes and azimuthals at time $s$ are $\theta_1(s),\ \theta_2(s)$, $\rho_1(s), \ \rho_2(s),\ \tau_1(s),$ and $\tau_2(s)$ respectively. A geodesic polygon is created from its equator of its complex ellipsoid with authalic radius $r(s)$ as 
\begin{align}\label{pol1}
\mathcal{A}(s)&= r^2(s)\ [\tau_2(s)-\tau_1(s)]+\int_{\rho_1}^{\rho_2}\int_{\theta_1}^{\theta_2}\left[\frac{1}{k(s)}-r^2(s)\right] \cos \theta(s)\ d\theta(s)\ d\rho(s),
\end{align}
where $k(s)$ is the Gaussian curvature created in the field and, $d\theta(s)$ and $ d\rho(s)$ are change in $\theta$ and $\rho$. If one plane has more than one type of curvature, then there are different ellipsoids corresponding to different curvatures whose $r,\tau_1,\tau_2,\theta_1,\theta_2,\rho_1$, and $\rho_2$ are different. If we add small arcs corresponding to each ellipsoid of each curvature we will get a total $\mathcal{A}_0(s)$, which is the area created by $\overrightarrow{PQ}$. If the strategy polygon is made of four geodesics $\overrightarrow{PQ}$, $\overrightarrow{QR}$, $\overrightarrow{RS}$ and $\overrightarrow{SP}$, the area of the strategy space can be calculated by the adding or subtracting the areas created by each of the geodesic lines, based on the locations of the equator of the complex ellipsoid.

\medskip

\begin{figure}
	\centering
	\includegraphics[width=15cm]{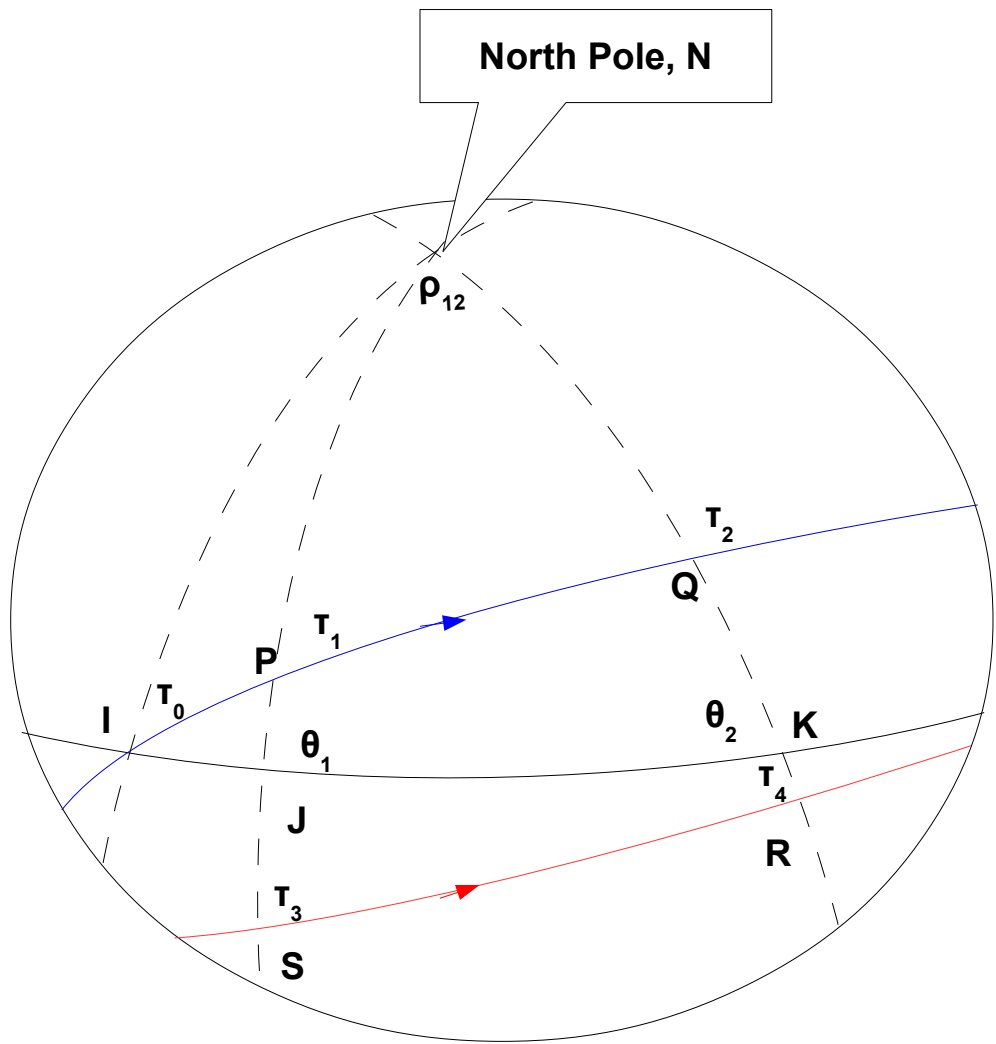}
	\caption{Geodesic strategy polygon $\square PQRS$ on an ellipsoid.}
	\label{p3}
\end{figure}

\medskip

Assume a firm has four strategies and its strategy polygon $\square PQRS$ has four sides as shown in Figure \ref{p3}. The lines $\overrightarrow{PQ}$ and $\overrightarrow{SR}$  are two geodesic lines on a ellipsoid with radius $r(s)$,  $\rho_{12}(s)=\rho_1(s)-\rho_2(s)$ is a longitudinal angle, $\tau_0(s),\tau_1(s),\tau_2(s),\tau_3(s)$, and $\tau_4(s)$ are azimuthal angles of $\overrightarrow{PQ}$ and $\overrightarrow{SR}$ respectively, and $\theta_1(s)$ and $\theta_2(s)$ are two latitudes on the equator $\overrightarrow{IJK}$. Using Equation (\ref{pol1}) we can calculate the area under $\overrightarrow{PQ}$ and the equator which is  $\square PQKJ$ and area over $\overrightarrow{SR}$ and the equator $\square JKRS$. Addition of these two areas is the total area of $\square PQRS$. On the other hand, if both $\overrightarrow{PQ}$ and $\overrightarrow{SR}$ are above the Equator $\overrightarrow{IJK}$ then $\square PQRS=\square PQKJ-\square SRKJ$. As the strategy set is always a convex, closed polygon with equal sides, we can calculate the area by adding or subtracting all the areas created by the geodesic lines forming the polygon.

\medskip

Let us define an indicator function $\mathds{1}_{ai}(s)$, where $a$ is the total number of equal sides of firm $i$'s strategy polygon at time $s$ such that,
\begin{equation}\label{ind}
\mathds{1}_{ai}(s)=\begin{cases} 1 & \text{if the geodesics are on the same side of the equator,}\\ 0    & \text{otherwise.}
\end{cases}
\end{equation}
Therefore, the area of the $i^{th}$ firm's closed, convex strategy polygon, is 
\begin{align}\label{ind0}
\mathcal{A}_i(s)&=\big[1-\mathds{1}_{ai}(s)\big] \sum_{a=1}^l\ \mathcal{A}_i^a(s)+\mathds{1}_{ai}(s) \left[\sum_{a=1}^{l_1} \mathcal{A}_i^a(s)- \sum_{a=1}^{l_2} \mathcal{A}_i^a(s)\right],
\end{align}
such that, $l_1+l_2=l$. In Equation (\ref{ind0}) $\mathcal{A}_i^a(s)$ is $i^{th}$ firm's area with the equator of the ellipsoid where $a$ is the total number of sides of the polygon at time $s$.

\medskip 

In a non-cooperative game firm $i$ has one strategy polygon $\mathcal{A}_i(s)$ and its strategy can not move beyond it. Contrarily, under semi-cooperation  one firm interacts with multiple other firms and has multiple  geodesic  polygons at different times. Therefore, at $s$ the firm can choose any strategy polygon with probability $\alpha_{1i}(s)\in[0,1]$ and its strategy $u_i(s)$ lies inside $\a_{1i}(s)\ \mathcal{A}_i(s)$. If there are $n$ firms in the economy, then under incomplete and imperfect information firm $i$ sees a portion $\a_{2j}(s)\in[0,1]$ of the strategy space of firm $j$. Therefore, firm $i$'s strategy depends on the interactions with its own $\a_{1i}^{\hat{\rho}}(s)\ \mathcal{A}_{i}(s)$ and $\sum_{j=1}^{n-1}\ \a_{2j}^{\tilde{\rho}}(s)\ \mathcal{A}_{j}(s)$, where $\a_{1i}^{\hat{\rho}}(s)$ is firm $i$'s probability to choose its strategy polygon $\mathcal{A}_{i}$ with the degree of cooperation ${\hat{\rho}}$ at time $s$ and $\a_{2j}^{\tilde{\rho}}(s)$ is firm $i$'s visualization of firm $j$'s strategy polygon at the same time with its realization of degree of cooperation $\tilde{\rho}$ where $(\hat{\rho},\tilde{\rho})\in(0,1]^2$. In other words, $\a_{1i}^{\hat{\rho}}(s)$ and  $\a_{2j}^{\tilde{\rho}}(s)$  capture the semi-cooperation and bounded rationality of firm $i$.

\medskip

Suppose that $p$ is a point in $\a_{1i}^{\hat{\rho}}(s)\mathcal{A}_{i}(s)$ which represents firm $i$'s behavior in its complex strategy polygon.

\begin{defn}(Tangent vector) Let $\Theta(\mathcal{M})_p$ be a local strategy ring at point $p$ on the $C^\infty$ complex manifold $\mathcal{M}$, then a tangent vector $V_p$ is a directional derivative at $p$ which is a linear map $V:\Theta(\mathcal{M})_p\ra\mathbb{R}$ that obeys Leibniz rule $V(f_1f_2)=V(f_1)f_2(p)+f_1(p)V(f_2)$ for two homeomorphic functions $f_1,f_2\in\Theta(\mathcal{M})_p$. 
\end{defn}

 We do not need any shape of manifold to define the directional derivative as long as the space is Hausdorff. Now we show the existence of curvature around the strategy of a firm with positive market share by considering the cotangent space built on gradients and a $1$-forms basis. The gradients and regular basis vectors are distinct in the sense that a basis is associated with the coordinate lines, but the gradients are associated with lines of steepest descent from one surface to another. Hence, gradients can be used to measure curvature in the strategy spacetime. Furthermore, we use the fact that the Lie bracket of two gradient vectors is non-zero when the space is curved. 
 
 \medskip
 
 Suppose that a firm $i$ in the strategy space with of market share $x_i(s)\in X$. Let $\boldsymbol{f}_i[s,x_i(s)]$ be a vector in $\mathcal{M}$ with gradients $ \boldsymbol{\nabla_V}=(v_\mu+b_\mu)\ \partial^\mu$ and $\boldsymbol{\nabla_U}=(u_\nu+w_\nu)\ \partial^\nu$ where $v_\mu,\ u_\nu$ are the drift parts of the gradients, and $b_\mu,\ w_\nu$ are the Brownian motion processes. Suppose that time is fixed and the vector $\boldsymbol{f}_i$ is moving along the strategy space, starting far from the strategy of the firm and moving closer to it. Initially the covariant components of the vector do not change the direction, but when the vectors approaches the firm, the direction of gradient vector changes. As we choose two different gradients the Lie bracket will never be zero around the firm.

\begin{prop}
Let $\boldsymbol{f}_i[s,x_i(s)]$ be the $i^{th}$ firm's vector on a complex Brownian strategy spacetime $\mathcal{M}$ with a measure of market share $x_i\in X$ at fixed time $s$. The two gradient vectors  $\boldsymbol{\nabla_V}$ and $\boldsymbol{\nabla_U}$ around the firm's strategy are $\boldsymbol{\nabla_V}=(v_\mu+b_\mu)\ \partial^\mu$, and $\boldsymbol{\nabla_U}=(u_\nu+w_\nu)\ \partial^\nu,$
where $v_\mu,\ u_\nu$ are drift components, $b_\mu,\ w_\nu$ are the Brownian motion covariant vector components, and $\partial^\mu,\ \partial^\nu$ are contravariant basis vectors of $1$-forms. Then firm $i$ creates a curvature around itself in $\mathcal{M}$.
\end{prop}

\begin{proof}
Suppose $\boldsymbol{f}_i[s,x_i(s)]$ is a vector representing the movement of the strategy of the $i^{th}$ firm at time $s$ with its measure of market share $x_i(s)\in X$ in a cotangent space $\mathcal{M}$. Since, $x_i(s)$ is a parameter of $\boldsymbol{f}_i$, we define another vector in the cotangent space $\boldsymbol{g}_i$ such that,
\begin{align}\label{g20}
\boldsymbol{g}_i(x_i) &= \boldsymbol{f}_i[y_\mu(x_i)],
\end{align}
where $y_\mu$ is a transformation of the coordinate system $y^\mu$ such that $y^\mu\rightarrow y_\mu$ with $\mu=1,2,3$, where the superscripts represent different axes of the same coordinate system. After differentiating both sides of Equation (\ref{g20}) with respect to $x_i$ we get,
\begin{align}
\frac{d\boldsymbol{g}_i}{dx_i}&=\sum_{\mu=1}^{3}\ \frac{dy_\mu}{dx_i}\frac{\partial \boldsymbol{f}_i}{\partial y_\mu}=\frac{dy_\mu}{dx_i}\frac{\partial \boldsymbol{f}_i}{\partial y_\mu},
\end{align}
using Einstein summation convention.

In Equation (\ref{g20}) $d\boldsymbol{g}_i/dx_i$ is cotangent or contravariant vector, $dy^\mu/dx_i$ is covariant components of the vector in the direction of the gradient vector of  $\boldsymbol{f}_i$, and $\partial \boldsymbol{f}_i/\partial y^\mu$ is the contravariant basis. Therefore, Equation (\ref{g20}) can be written as $\boldsymbol{\nabla_V}_i=v_\mu \partial^\mu \boldsymbol{f}_i$. As we are only considering only one firm, we ignore subscript $i$ and write down the gradient vector representation of Equation (\ref{g20}) as $\boldsymbol{\nabla_V}=v_\mu \partial^\mu \boldsymbol{f}$ where $\boldsymbol{\nabla_V}=d\boldsymbol{g}_i/dx_i$, $v_\mu=dy_\mu/dx_i$ and $\partial^\mu \boldsymbol{f}=\partial \boldsymbol{f}_i/\partial y_\mu$. As we have Brownian strategy spacetime, the vectors in this field are differentiable in the direction of the strategy space but non-holomorphic everywhere. We know that for any given manifold the basis vector does not change, only the covariant vector component $v_\mu$ changes. 

Suppose there are two gradients $\boldsymbol{\nabla_V}$ and $\boldsymbol{\nabla_U}$ on Brownian strategy surface such that $\boldsymbol{\nabla_V}=(v_\mu+b_\mu)\ \partial^\mu$, and $\boldsymbol{\nabla_U}=(u_\nu+w_\nu)\ \partial^\nu,$ where $v_\mu,\ u_\nu$ are drift components and $b_\mu,\ w_\nu$ are the Brownian motion components of gradients $\boldsymbol{\nabla_V}$ and $\boldsymbol{\nabla_U}$ respectively are differentiable with respect to market share $x$. Finally, $\partial^\nu$ is the contravariant basis vector with respect to the direction of  $\boldsymbol{\nabla_U}$.

Therefore, the Lie bracket or commutator of the two gradients of $\boldsymbol{f}$ can be expressed as
\begin{align}\label{g23}
[\boldsymbol{\nabla_V},\boldsymbol{\nabla_U}]\boldsymbol{f}&=[v_\mu\ \partial^\mu,\ u_\nu\ \partial^\nu]\boldsymbol{f}+[v_\mu\ \partial^\mu,\ w_\nu\ \partial^\nu]\boldsymbol{f}+[b_\mu\ \partial^\mu,\ u_\nu\ \partial^\nu]\boldsymbol{f}+[b_\mu\ \partial^\mu,\ w_\nu\ \partial^\nu]\boldsymbol{f}.
\end{align}
We will calculate the components for each of the four Lie brackets using the contraction property of tensors. First Lie bracket of the right hand side of Equation (\ref{g23}) is 
\begin{align}\label{g24}
[v_\mu\ \partial^\mu,\ u_\nu\ \partial^\nu]\boldsymbol{f}&=v_\mu\ \partial^\mu[u_\nu\ \partial^\nu\boldsymbol{f}]-u_\nu\ \partial^\nu[v_\mu\ \partial^\mu\boldsymbol{f}]\notag\\&=v_\mu\ \partial^\mu[u_\nu\ \partial^\nu\boldsymbol{f}]-u_\mu\ \partial^\mu[v_\nu\ \partial^\mu\boldsymbol{f}],\notag\\&=[v_\mu\ \partial^\mu u_\nu-u_\mu\ \partial^\mu v_\nu]\ \partial^\nu\boldsymbol{f}.
\end{align}
Similarly, 
\begin{align}\label{g25}
[v_\mu\ \partial^\mu,\ w_\nu\ \partial^\nu]\boldsymbol{f}&=[v_\mu\ \partial^\mu w_\nu-w_\mu\ \partial^\mu v_\nu]\ \partial^\nu\boldsymbol{f},\notag\\ [b_\mu\ \partial^\mu,\ u_\nu\ \partial^\nu]\boldsymbol{f}&=[b_\mu\ \partial^\mu u_\nu-u_\mu\ \partial^\mu b_\nu]\ \partial^\nu\boldsymbol{f},
\end{align}
and
\begin{align}\label{g25.0}
 [b_\mu\ \partial^\mu,\ w_\nu\ \partial^\nu]\boldsymbol{f}&=[b_\mu\ \partial^\mu w_\nu-w_\mu\ \partial^\mu b_\nu]\ \partial^\nu\boldsymbol{f}.
\end{align}

Equations (\ref{g24})- (\ref{g25.0}) imply,
\begin{align}\label{g26}
[\boldsymbol{\nabla_V},\boldsymbol{\nabla_U}]\boldsymbol{f}&=\biggr[\left(v_\mu\frac{\partial u_\nu}{\partial y_\mu}-u_\mu\frac{\partial v_\nu}{\partial y_\mu}\right)+\left(v_\mu\frac{\partial w_\nu}{\partial y_\mu}-w_\mu\frac{\partial v_\nu}{\partial y_\mu}\right)\notag\\&\hspace{1cm}+\left(b_\mu\frac{\partial u_\nu}{\partial y_\mu}-u_\mu\frac{\partial b_\nu}{\partial y_\mu}\right)+\left(b_\mu\frac{\partial w_\nu}{\partial y_\mu}-w_\mu\frac{\partial b_\nu}{\partial y_\mu}\right)\biggr]\ \partial^\nu\boldsymbol{f}\neq 0,
\end{align}
as all the covariant components of $\boldsymbol{\nabla_V}$ and $\boldsymbol{\nabla_U}$ are different. Hence, a firm with positive market share creates a curvature around its strategy in the strategy space.
\end{proof}

\begin{remark}
Only the rationality of a firm can create a positive measure of market share and hence a curvature. This type of argument can be shown rigorously with a more generalized curved strategy spacetime which is asymptotically Minkowski using an argument similar to the Riemann-Penrose conjecture (\cite{schoen1979}, \cite{witten1981}).
\end{remark}

\section{Construction of the field of strategy spacetime}

\medskip
Represent the dynamic stochastic discounted profit function of the $i^{th}$ firm as $\pi_i[s,x_i(s),u_i[\a_{1i}^{\hat{\rho}}(s) \mathcal{A}_i(s)]]$, where $x_i$ and $u_i[\a_{1i}^{\hat{\rho}}(s) \mathcal{A}_i(s)]$ are firm $i$'s measure of market share and strategy at time $s$ respectively, where firm $i$'s strategy is a monotonic function of its $\a_{1i}^{\hat{\rho}}(s) \mathcal{A}_i(s)$. It is quadratic in terms of change in time or $(d/ds)^2$, non-decreasing in output price, non-increasing in input price, homogenous of degree one in output and input prices, convex and continuous in output and input prices. 

\begin{definition}\label{d2}
(Schwartz space) Suppose for $i=1,2,...,k,$ $\pi_{i1}$ is a $C^{\infty}$ function over $\mathbb{R}^{n\times k}$, then define
\begin{align}
\mathfrak{S}(\mathbb{R}^{n\times k})&=\{\pi_{i1}\in C^\infty(\mathbb{R}^{n\times k}):\ ||\pi_{i1}||_{\a,\be}<\infty,\ \forall\ \a,\be\in\mathbb{N}^{n\times k}\ \text{and,}\ i=1,2,...,k.\},\notag
\end{align}
where $\a,\ \be$ are multi-indices, $C^\infty(\mathbb{R}^{n\times k})$ is a set of smooth functions from $\mathbb{R}^{n\times k}$ to the complex space $\mathbb{C}$ (i.e. $C^\infty(\mathbb{R}^{n\times k}):\mathbb{R}^{n\times k}\ra \mathbb{C}$) and for an element $r$ in  $\mathbb{R}^{n\times k}$ 
\begin{align}
||\pi_{i1}||_{\a,\be}&:=\sup_{r\in\mathbb{R}^{n\times k}}\ |r^\a\ D^\be\ \pi_{i1}(r)|,\notag
\end{align}
where $D^\be\ \pi_{i1}(r)$ is $\be^{th}$ order derivative of $\pi_{i1}(r)$ and $r=\{i,x,u\}$.
\end{definition}

\medskip

\begin{rmk}\label{r0}
In our case, as time is imaginary, the real dimensions of the strategy spacetime is $2$, as $\pi_{i1}\in \mathfrak{S}(\mathbb{R}^{n\times k})$ it is a test function. 
\end{rmk}

Based on Definition \ref{d2} we assume there exists a $C^\infty$ function $\pi_{i1}$ which can replace the stochastic profit $\pi_i$ by It\^o's Theorem is in the Schwartz space.

\begin{defn}\label{d3}
(Distribution) For a Schwartz space $\mathfrak{S}(\mathbb{R}^{n\times k})$ and for all $a,m\in\mathbb{N}$ if the function $\pi_{i1}:\mathbb{R}^{n\times k}\ra \mathbb{C}$ satisfies  continuity with respect to all the semi-norms,
\begin{align}
|\pi_{i1}|_{a,m}:=\ \sup_{|\be|\leq a}\ \sup_{r\in\mathbb{R}^{n\times k}}\  |\partial^\be\pi_{i1}(r)|\ (1+|r|^2)^m<\infty,\notag 
\end{align}	
then $\pi_{i1}$ is called a distribution. 
\end{defn}

\begin{defn}\label{d4}
(Tempered distribution)  For a Schwartz space $\mathfrak{S}(\mathbb{R}^{n\times k})$ and for all $a,m\in\mathbb{N}$ if the linear function $\mathcal{T}:\mathfrak{S}(\mathbb{R}^{n\times k})\ra \mathbb{C}$ satisfies  continuity with respect to all the semi-norms,
\begin{align}
|\pi_{i1}|_{a,m}:=\ \sup_{|\be|\leq a} \ \sup_{r\in\mathbb{R}^{n\times k}}\  |\partial^\be\pi_{i1}(r)|\ (1+|r|^2)^m<\infty,\notag 
\end{align}	
then $\mathcal{T}$ is called a tempered distribution.
\end{defn}

\begin{defn}\label{d5}
(Operator valued distribution) Assume our strategy spacetime $\mathcal{S}$ is a three dimensional topological vector space on a smooth manifold  $\mathcal{M}$. An operator valued distribution on $\mathcal{M}$ is a continuous linear functional $\Psi$ such that $\Psi:C_c^\infty(\mathcal{M})\ra \mathcal{S}$, where $C_c^\infty(\mathcal{M})$ is the space of smooth bump functions with compact support.
\end{defn}

\begin{rmk}\label{r1}
If we replace $C_c^\infty(\mathcal{M})$ by $C^\infty(\mathcal{M})$ then Definition \ref{d5} becomes the definition of a compactly supported distribution. Replacing $C_c^\infty(\mathcal{M})$ by the Schwartz space $\mathfrak{S}(\mathcal{M})$ connects  Definitions \ref{d4} and \ref{d5}.
\end{rmk}

We assume the strategy space $\mathcal{S}$ is curved around the strategy of a firm with market share $x_i(s)\ra 1$ otherwise, the space is a Minkowski flat strategy spacetime which follows Garding-Wightman axioms in quantum field theory \cite{wightman1965}.

\begin{ax}\label{ax1}
(Garding-Wightman axioms of strategy spacetime)
\begin{itemize}
\item For any two coordinate systems $y$ and $y'$ such that $y\ra y'=e+\Lambda y$ there exists a physical Hilbert space $\mathbb{H}$ where a unitary representation $\mathbb{U}(e,\Lambda)$ with respect to the translation $e\in\mathbb{R}^{3}$ and Lorentz transformation $\Lambda_{3\times 3}$ of the Poincar\'e spinor group, $P_0$ acts.

\item The wave paths of the energy of the strategy of the firm  $P$ are concentrated in a closed upper forward light cone in Kruskal-Szekeres coordinate.

\item If a monopolist firm does not have any constraint imposed by the market, then  this firm can achieve a tremendous amount of energy to move in the strategy spacetime because of the Casimir-Polder effect \cite{casimir1948}. In reality this may not be achievable because of market dynamics and the existence of other firms which work in the opposite direction of this energy, such that the strategy field exerts very low energy. This property is equivalent to the vaccum state (the ket vector with zero potential $|0\rangle$) and in this case a unique unit vector $|0\rangle$ exists in $\mathbb{H}$ such that it is invariant to the spacetime translation $\mathbb{U}(e,1)$.

\item The components $\Psi_i$ of the quantum strategy spacetime are operator valued distributions over $\mathfrak{S}(\mathcal{M})$ with domain $\mathbb{D}$, which is common all the operators and is dense in $\mathbb{H}$.

\item $\mathbb{U}(e,\Lambda) \Psi_i(r)\mathbb{U}(e,\Lambda)^{-1}=\sum_j W_{ij}(\Lambda^{-1}) \Psi_j(e+\Lambda r)$, where $W_{ij}(\Lambda^{-1})$ is a complex or real finite dimensional matrix representation of the M\"obius map $\text{SL}(2,C)$.

\item For any space-like separation (two dimensions of strategy spacetime without imaginary time) of arguments $r_1$ and $r_2$ any two field components $\Psi_i(r_1)$ and $\Psi_j(r_2)$ either commute or anticommute.

\item Define a set $\mathbb{D}_0$ as a finite linear combinations of vectors of the form $\Psi_{i1}(\pi_{i1}),...,\Psi_{ik}(\pi_{k1})|0\rangle$. This combination is dense in $\mathbb{H}$. 
\end{itemize}
\end{ax}

\begin{as}\label{as1}
(Assumptions on the background manifold $\mathcal{M}$)
\begin{itemize}
\item For multi-indices $\gamma,\be$ there exists a family $\mathscr{F}=\{F_\gamma, \phi_\gamma\}$ of coordinate neighborhoods on the manifold $\mathcal{M}$ such that,\\ (i) $F_\gamma$ covers $\mathcal{M}$,\\ (ii) For any $\gamma$ and $\be$ the neighborhoods $F_\gamma$, $\phi_\be$ and $F_\be$, $\phi_\gamma$ are $C^\infty$-compatible and,\\ (iii) For any coordinate neighborhood $F, \phi$ compatible with every $F_\gamma, \phi_\gamma\in\mathscr{F}$ is itself in $\mathscr{F}$, then $\mathcal{M}$ is a smooth manifold \cite{boothby1986}.

\item The strategy of firm $i$ moves on smooth manifold $\mathcal{M}$ with its geodesic polygon $\a_{1i}^{\hat{\rho}}(s) \mathcal{A}_i(s)$ where the area of the polygon $\mathcal{A}_i(s)$ is defined in Equation (\ref{ind0}).

\item For all positive measures of market share of firm $i$, the strategy spacetime is either curved (Einstein metric) or, it is Minkowski flat and satisfies Axiom \ref{ax1}.

\item If firm $i$ does not want to cooperate with others, its strategy can move freely along the Minkowski strategy spacetime over time creating a curvature around the strategy based on the value of measure of its market share.

\item Inside a geodesic strategy polygon $\mathcal{A}_i(s)$  firm  $i$ faces a curved random surface (Liouville-like quantum gravity surface) which can be approximated by a Brownian surface \cite{miller2015liouville}.
\end{itemize}	
\end{as}

From  Assumption \ref{as1} and Axiom \ref{ax1} we know that the background manifold $\mathcal{M}$ is smooth and it is curved around the strategy of some firm with a positive measure of market share and is flat otherwise. In the flat case we consider the Minkowski metric $\eta_{\theta\nu}(U)$ for two co-ordinate systems $\theta$ and $\nu$ of the functional strategy space $U$ and for the curved case we consider the Einstein metric $G_{\theta\nu}(U)$ defined as $G_{\theta\nu}(U)=R_{\theta\nu}(U)-\frac{1}{2}Rg_{\theta\nu}(U)$, where $R_{\theta\nu}(U)$ is the Ricci curvature tensor, $R$ is Ricci scalar and $g_{\theta\nu}(U)$ is Riemann metric tensor on $U$. For a flat strategy spacetime $R_{\theta\nu}(U)-\frac{1}{2}Rg_{\theta\nu}(U)=0$ and therefore $G_{\theta\nu}(U)$ vanishes and takes the value of $\eta_{\theta\nu}(U)$. The Brownian surface is constructed from random planner maps based on quadrangulations and the Schramm-Loewner evolution (SLE) equation coincides with certain types of Liouville Quantum Gravity (LQG) surfaces (\cite{duplantier2011}, \cite{le2013}, \cite{mavromatos1989} and \cite{miermont2013}). Therefore, we directly use this surface to model the movement of the strategy of firm $i$ in its geodesic polygon $\mathcal{A}_i(s)$.

\medskip

As future strategies of a firm depend on the decisions from the governing bodies, their stubbornness takes an important part. If the governing body is too stubborn, then their decision is invariant with the change of the market environment. On the other hand, if they are very flexible, firm's reaction is very sensitive to the market environment. Let us define a homeomorphic function $\gamma$ such that, $\gamma: \mathfrak{s}\ra[0,2]$ where $\mathfrak{s}$ is a space of stubbornness where,
\begin{as}\label{as2}
 (i) The empty set $\emptyset\in\mathfrak{s}$ and for an element $\mathfrak{b}\in\mathfrak{s}$ the complement $\mathfrak{b}^c\in\mathfrak{s}$. \\ (ii)  $\mathfrak{s}$ is a convex, complete and bounded measurable set.\\ (iii) For ordered stubbornness levels $\mathfrak{b}_{[1]}, \mathfrak{b}_{[2]},..., \mathfrak{b}_{[n]}\in\mathfrak{s}$, we have $\bigcup_{i=1}^n \mathfrak{b}_{[i]} \in\mathfrak{s}$ such that $\gamma\left(\mathfrak{b}_{[1]}\right)=0$ and $\gamma\left(\mathfrak{b}_{[n]}\right)=2$.\\ (iv) For $i\neq j$ we have $\mathfrak{b}_{[i]}\bigcap\mathfrak{b}_{[j]}=\emptyset$. \\ (v) For $i=1,2,... ,n$, $\gamma\left(\bigcup_{i=1}^n \mathfrak{b}_{[i]}\right)=\sum_{i=1}^n \gamma\left(\mathfrak{b}_{[i]}\right)$.	
\end{as}

\medskip

\begin{defn}\label{d6}
(Stubbornness distribution)
For any domain $D$ in the complex strategy spacetime $U$, $z\in D$,  $\gamma:\mathfrak{s}\ra[0,2]$ and for a Riemann metric tensor $g_{\theta\nu}(U)$ the stubbornness distribution $\mathfrak{B}$ is defined as $\mathfrak{B}=e^{\gamma\ b(z)}\ g_{\theta\nu}(U)$, where the real valued function $b(z):\mathbb{C}\ra\mathbb{R}$ is a Gaussian free field and is defined by the Dirichlet inner product $\boldsymbol{\nabla}(f_1,\ f_2)(z)=(2\pi)^{-1}\int_D\ \boldsymbol{\nabla} f_1(z).\ \boldsymbol{\nabla}f_2(z)\ d^2z$, where $\boldsymbol{\nabla}$ is the gradient of the two  orthonormal bases $f_1$ and $f_2$ and, $d^2z$ is Lebesgue measure.
\end{defn}

\begin{remark}\label{r2}
Following \cite{duplantier2011}, \cite{knizhnik1988} and \cite{polyakov1981}, if $\gamma\ra 2$ the random space becomes flat with a few peaks which are termed as baby universes. Hence, at $\gamma\ra 2$ the governing body has extreme stubbornness and at $\gamma\ra 0$ the surface becomes random and hence, more flexible. Furthermore, if $\gamma=\sqrt{8/3}$, then the strategy spacetime becomes a Brownian surface (\cite{gwynne2016}, \cite{miller2015} and \cite{miller2015liouville}).
\end{remark}

\medskip

As a firm of zero market share does not exist in the strategy spacetime, the Minkowski metric $\eta_{\theta\nu}(U)$ does not require a multiplicative coefficient. On the other hand, as firms with positive market shares exist, the Einstein metric $G_{\theta\nu}(U)$ is multiplied by $e^{\gamma b(z)}$. The metric of the strategy spacetime is defined by $N_{\theta\nu}(U)=e^{\gamma b(z)}G_{\theta\nu}(U)+\eta_{\theta\nu}(U)$. The measure of market share of firm $i$ traversing time can be represented by an open string. A path over time in $U$ is ergodic and the history of a string should be in a torus-like area in our three dimensional strategy spacetime which is a world volume $\Sigma$ where $\sigma^{\mathscr{B}}=(s,\sigma_1,\sigma_2)$ are the three coordinates on the world-volume with $0\leq s\leq t$, $\sigma_{1i}<\sigma_1<\sigma_{1f}<\infty$, and $\sigma_{2i}<\sigma_2<\sigma_{2f}<\infty$, where $\sigma_{1i}$ and $\sigma_{1f}$ represent the initial and final values of $\sigma_1$ respectively. Similarly, $\sigma_{2i}$ and $\sigma_{2f}$ are initial and final values of $\sigma_2$ respectively.  The real valued function $b(s,\sigma_1,\sigma_2)$ is a map from the complex world-volume into the $3$-dimensional random space whose metric is $N_{\theta\nu}(U)$. In order to avoid the square root problem in Nambu-Goto  action (\cite{goto1971}, \cite{nambu1970}) we introduce an additional (auxiliary) field on the strategy spacetime and is represented as a world volume metric $h_{\tau_1\tau_2}(s,\sigma_1,\sigma_2)$ with signature $(-,+,+)$ . Assume $h^{\tau_1\tau_2}$ is an inverse and $h$ is the determinant of $h_{\tau_1\tau_2}$.

\medskip

For a dynamic continuously integrable profit function $\pi_i[s,x_i(s,\sigma_1,\sigma_2),u_i[\a_{1i}^{\hat{\rho}}(s) \mathcal{A}_i(s,\sigma_1,\sigma_2)]]$ firm $i$'s objective is to find $u$ that satisfy
\begin{align}\label{obj}
\max_{u\in U} \overline{\Pi}(u,t)&=\max_{u\in U} \int_0^t\int_{\sigma_{1i}}^{\sigma_{1f}}\int_{\sigma_{2i}}^{\sigma_{2f}}\ \biggr\{ \E \pi_i\big[s,x_i(s,\sigma_1,\sigma_2),u_i[\a_{1i}^{\hat{\rho}}(s) \mathcal{A}_i(s,\sigma_1,\sigma_2)]\big]\biggr\} b_i ds d\sigma_1 d\sigma_2\notag\\&=\max_{u\in U} \int_\Sigma\ \biggr\{ \E\pi_i\big[s,x_i(s,\sigma_1,\sigma_2),u_i[\a_{1i}^{\hat{\rho}}(s) \mathcal{A}_i(s,\sigma_1,\sigma_2)]\big]\biggr\} b_id\sigma^3,
\end{align}
subject to the market dynamics  
\begin{align}
dx_i(s,\sigma_1,\sigma_2)&=\mu_i^{\tau_1}[s,x_i(s,\sigma_1,\sigma_2),u_i[\a_{1i}^{\hat{\rho}}(s) \mathcal{A}_i(s,\sigma_1,\sigma_2)]] ds\notag\\&\hspace{.5cm}+\sum_{k=1}^3\ \omega_{i,k}^{\tau_1}[s,x_i(s,\sigma_1,\sigma_2),u_i[\a_{1i}^{\hat{\rho}}(s) \mathcal{A}_i(s,\sigma_1,\sigma_2)]] dB^k(s,\sigma_1,\sigma_2),
\end{align}
 with initial condition $x_i(0,\sigma_{1i},\sigma_{2i})=x_0^i$ for all $i=1,...,n$, where $b_i$ is a real valued stubbornness function of the governing body of $i^{th}$ firm, the area of $\mathcal{A}_i(s,\sigma_1,\sigma_2)$ is a convex geodesic polyhedron and $B^k(s,\sigma_1,\sigma_2)$ is a $3$-dimensional Brownian motion for each of $n$ firms. In Equation (\ref{obj}) the profit function $\pi_i=\pi_i^+-\pi_i^-$ is the combination of profit creation  $\pi_i^+$ and profit reduction $\pi_i^-$ operators which increase and decrease the profit by at least one unit respectively so that the profit state of firm $i$ at time $s$ is $\widetilde{\pi}_i=\pi_i^+\pm\iota\pi_i^-$, where $\iota$ is the imaginary unit. For two firms $i$ and $j$ these operators follow the Lie commutator bracket properties $[\pi_i^-,\pi_j^+]\equiv\delta_{ij}$ and $[\pi_i^+,\pi_j^+]=[\pi_i^-,\pi_j^-]=0$, where $\delta_{ij}$ is the Kronecker delta function for $i,j=1,2,...,n$. The imaginary number $\iota$ helps the dynamic profit function to move in the strategy space in a circular way. Without $\iota$, $\widetilde{\pi}_i$ moves along only left and right directions. The difference between $\pi_i$ and $\widetilde{\pi}_i$ is that, first one only the amount of contribution of each of the operators in determining firm $i$'s profit where $\widetilde{\pi}_i$ invokes a special dynamic movement towards the left or right. If $\widetilde{\pi}_i$ moves towards positive direction of time then $\pi_{i}^+$ has more contribution than $\pi_{i}^-$, and vice versa. If $\pi_i^-$ contributes more at first $\widetilde\pi_i$ moves backwards in time and firm $i$ will shut down at $\pi_i=0$.
 
 \begin{lem}\label{l}
 At a given fixed time $s$, if firm $i$ sells the rights of its products (bonds) to its very large $m\in\mathbb{N}$ identical consumer base ($m\ra\infty$) multiple times and for any asymmetric information $\kappa>0$, each consumer creates an exponential negative effect $\exp(-\kappa)$ on $\pi_i^-$ and creates an infinite negative effect in strategy spacetime as $\kappa\ra 0$.
 \end{lem}
 
 \begin{proof}
 Suppose, firm $i$ has $m$ consumers and sells bonds of a product $\theta$ times to its $j^{th}$ consumer. By assumption, each consumer is identical and when firm $i$ sells the bond of the product to the $j^{th}$ consumer for the $k^{th}$ time it creates an exponential negative effect $\exp(-k\kappa)$. We assume $\kappa\ra 0$ because consumer $j$ is small compared to  firm $i$ so that its individual effect is negligible.
 
 Suppose that firm $i$ sells a bond of a product  to the $j^{th}$ consumer $\theta_j$ times such that there is an addition of one unit of contribution to $\pi_i^-$ for each sale. The total contribution from consumer $j$ will be $\mathfrak{C}_j=\sum_{{\rho_j}=1}^{\theta_j} \rho_j\pi_i^-$. As $\kappa\ra0$, the negative effect of each sale of the bond to consumer $\rho$ can be written as $\mathfrak{C}_j=\sum_{\rho_j=1}^{\theta_j} \exp(-\rho_j\kappa) \rho_j\pi_i^-$. 
 
 At a fixed time $s$ we can use the geometric series 
 \begin{align}\label{b1}
 \sum_{\rho_j=1}^{\theta_j} \exp(-\rho_j\kappa)\pi_i^-&=\frac{1-\exp(-\theta_j\kappa)}{1-\exp(-\kappa)}\pi_i^-.
 \end{align}
  Differentiating both sides of Equation (\ref{b1}) with respect to $\exp(-\kappa)$ yields,
  \begin{align}\label{b2}
  \sum_{\rho_j=1}^{\theta_j} \rho_j [\exp(-\kappa)]^{\rho_j-1}\ \pi_i^-&=\pi_i^-\left[\frac{1-\exp[-(\theta_j+1)\kappa]}{[1-\exp(-\kappa)]^2}-\frac{(\theta_j+1)\exp(-\theta_j\kappa)}{1-\exp(-\kappa)}\right].
  \end{align}
  For all $m$ consumers the total effect is
  \begin{align}\label{b3}
  \sum_{j=1}^m\sum_{\rho_j=1}^{\theta_j} \rho_j [\exp(-\kappa)]^{\rho_j-1}\ \pi_i^-&=\sum_{j=1}^m\ \pi_i^-\left[\frac{1-\exp[-(\theta_j+1)\kappa]}{[1-\exp(-\kappa)]^2}-\frac{(\theta_j+1)\exp(-\theta_j\kappa)}{1-\exp(-\kappa)}\right].
  \end{align}
Letting $m\ra\infty$,  Equation (\ref{d3}) becomes,
\begin{align}\label{b4}
 \lim_{m\ra\infty}\ \sum_{j=1}^m\sum_{\rho_j=1}^{\theta_j} \rho_j [\exp(-\kappa)]^{\rho_j-1}\ \pi_i^-&\approx  \frac{\pi_i^-}{[1-\exp(-\kappa)]^2}=\frac{\pi_i^-}{1-2\exp(-\kappa)+\exp(-2\kappa)}.
\end{align}
As $\kappa\ra 0$, $\exp(-2\kappa)\approx \exp(-\kappa)$, so that
\begin{align}\label{b5}
\lim_{m\ra\infty}\ \sum_{j=1}^m\sum_{\rho_j=1}^{\theta_j} \rho_j [\exp(-\kappa)]^{\rho_j-1}\ \pi_i^-&\approx  \frac{\pi_i^-}{1-\exp(-\kappa)}.
\end{align}
Multiplying $\exp(-\kappa)$ into both sides of Equation (\ref{b5}) we get,
\begin{align}\label{b6}
\lim_{m\ra\infty}\ \sum_{j=1}^m \mathfrak{C}_j=\lim_{m\ra\infty} \exp(-\kappa) \sum_{j=1}^m\sum_{\rho_j=1}^{\theta_j} \rho_j [\exp(-\kappa)]^{\rho_j-1} \pi_i^-&=  \frac{ \exp(-\kappa)}{1-\exp(-\kappa)}\ \pi_i^-.
\end{align}
After expanding the exponential series in the numerator and denominator of the right hand side and differentiating both sides with respect to $\kappa$, Equation(\ref{b6}) becomes,
\begin{align}
\frac{\partial}{\partial\kappa}\left(\lim_{m\ra\infty}\ \sum_{j=1}^m \mathfrak{C}_j\right)& \approx \pi_i^- \frac{\partial}{\partial\kappa}\left\{ \frac{1}{\kappa}\left[1-\frac{\kappa}{2}+\frac{\kappa^2}{12}\right]\right\}=-\frac{1}{\kappa^2}+\frac{1}{12}.\notag
\end{align}
 As we assumed $\kappa\ra0$ then $\frac{\partial}{\partial\kappa}\left(\lim_{m\ra\infty}\ \sum_{j=1}^m \mathfrak{C}_j\right)=-\infty$. Therefore, firm $i$ creates an infinite negative effect by selling bonds repeatedly to its consumers.
 \end{proof}
 
 \begin{remark}
 Lemma \ref{l} might explain the $2007$ housing market crisis in the United States where housing bonds were sold multiple times to the same customer and created an infinite negative effect on $\pi_i^-$ such that companies like Lehman Brothers faced a complete shut down.
 \end{remark}
 \medskip
 
Following \cite{ito1962}, for any $3$-dimensional manifold $\mathcal{M}$ we can write the ${\tau_1}^{th}$-dimensional drift component of $i^{th}$ firm as $\mu_i^{\tau_1}=-1/2\  _ih^{k_1 k_2}\ _i\Gamma_{k_1 k_2}^{\tau_1}$, where $_i\Gamma_{k_1 k_2}^{\tau_1}$ is the Christoffel symbol defined as,
\begin{align}\label{ch}
_i\Gamma_{k_1 k_2}^{\tau_1}&=\mbox{$\frac{1}{2}$}\ _ih^{\tau_1\delta}\biggr(\ _ih_{\delta k_2;\ k_1}+\ _ih_{\delta k_1;\ k_2}-\ _ih_{k_1 k_2;\ \delta}\biggr),
\end{align}
where $\tau_1,\delta,k_1,k_2=\{1,2,3\}^4$,  $_ih_{\delta k_2;\ k_1}=\partial_{k_1}^i h_{\delta k_2}$, $_ih_{\delta k_1;\ k_2}=\partial_{k_2}^i h_{\delta k_1}$ and $_ih_{k_1 k_2;\ \delta}=\partial_{\delta}^i h_{k_1 k_2}$. Furthermore, we can write $\sum_{k=1}^3\ \omega_{i,k}^{\tau_1}\ \omega_{i,k}^{\tau_2}=\  _ih^{\tau_1\tau_2}$. 

\medskip

\begin{as}\label{as3}
For any $3$-dimensional manifold $\mathcal{M}$ and for constants $0<\mathfrak{A}<\infty$ and $0<\mathfrak{B}<\infty$,  locally Euclidean space satisfies,
\begin{align}
&\sum_{\tau_1=1}^3\ \big|\mu_i^{\tau_1}[s,x_i(s,\sigma_1,\sigma_2),u_i[\a_{1i}^{\hat{\rho}}(s)\ \mathcal{A}_i(s,\sigma_1,\sigma_2)]]-\mu_i^{\tau_1}[s,y_i(s,\sigma_1,\sigma_2),u_i[\a_{1i}^{\hat{\rho}}(s)\ \mathcal{A}_i(s,\sigma_1,\sigma_2)]]\big|^2\notag\\&\hspace{1cm}\leq\mathfrak{A}\ \big|\big|x_i(s,\sigma_1,\sigma_2)-y_i(s,\sigma_1,\sigma_2)\big|\big|^2,\notag
\end{align}
and,
\begin{align}
&\sum_{k=1}^3\ \sum_{\tau_1=1}^3\ \big| \omega_{i,k}^{\tau_1}[s,x_i(s,\sigma_1,\sigma_2),u_i[\a_{1i}^{\hat{\rho}}(s)\ \mathcal{A}_i(s,\sigma_1,\sigma_2)]]- \omega_{i,k}^{\tau_1}[s,y_i(s,\sigma_1,\sigma_2),u_i[\a_{1i}^{\hat{\rho}}(s)\ \mathcal{A}_i(s,\sigma_1,\sigma_2)]]\big|^2\notag\\&\hspace{1cm}\leq \mathfrak{B}\ \big|\big| x_i(s,\sigma_1,\sigma_2)-y_i(s,\sigma_1,\sigma_2)\big|\big|^2,\notag
\end{align}
there exists $\delta>0$ such that, for all $|x-y|<\delta$, that locally satisfies,
\begin{align}
\lim_{\delta\ra 0}\big|\big| x_i(s,\sigma_1,\sigma_2)-y_i(s,\sigma_1,\sigma_2)\big|\big|^2=\sum_{\tau_1=1}^3\ \big|x_i^{\tau_1}-y_i^{\tau_1}\big|^2,\notag
\end{align}
where $\mu_i^{\tau_1}$, $\omega_{i,k}^{\tau_1}$ are continuous in time and for any $x$ and finite positive constants $\mathfrak{A}_0,\mathfrak{B}_0$, locally
\begin{align}
\sum_{\tau_1=1}^3\ \big|\mu_i^{\tau_1}[s,x_i(s,\sigma_1,\sigma_2),u_i[\a_{1i}^{\hat{\rho}}(s)\ \mathcal{A}_i(s,\sigma_1,\sigma_2)]]\big|^2\leq\mathfrak{A}_0,\notag
\end{align}
and
\begin{align}
\sum_{k=1}^3\ \sum_{\tau_1=1}^3\ \big| \omega_{i,k}^{\tau_1}[s,x_i(s,\sigma_1,\sigma_2),u_i[\a_{1i}^{\hat{\rho}}(s)\ \mathcal{A}_i(s,\sigma_1,\sigma_2)]]\big|^2\leq\mathfrak{B}_0.\notag
\end{align}
\end{as}

\medskip

\begin{defn}\label{Nas}
(Cooperative Nash Equilibrium) 	A set of optimal strategies \\ $\{u_1^*(s,\sigma_1,\sigma_2),...,u_n^*(s,\sigma_1,\sigma_2)\}$ forms a cooperative Nash equilibrium of an n-firm differential game if 
\begin{align}\label{nas10}
&\E\ \left\{\int_\Sigma\ \pi_i\left[s,x_i^*(s,\sigma_1,\sigma_2), u_i^*[\a_{1i}^{\hat{\rho}}(s)\ \mathcal{A}(s,\sigma_1,\sigma_2)],u_{-i}^*[\a_{2,-i}^{\tilde{\rho}}(s)\ \mathcal{A}_{-i}(s,\sigma_1,\sigma_2)]\right] b_i\ d\sigma^3\right\}\notag\\&\hspace{1cm}\geq\ \E\ \left\{\int_\Sigma\  \pi_i[s,x_i(s,\sigma_1,\sigma_2), u_i[\a_{1i}^{\hat{\rho}}(s)\ \mathcal{A}_i(s,\sigma_1,\sigma_2)],u_{-i}^*[\a_{2,-i}^{\tilde{\rho}}(s)\ \mathcal{A}_{-i}(s,\sigma_1,\sigma_2)]] \ b_i\ d\sigma^3\right\},
\end{align}
for all $i\in\{1,...,n\}$ where $t\in(0,\infty)$,
subject to the constraints,
\begin{align}\label{nas11}
dx_i^*(s,\sigma_1,\sigma_2)&=\mu_i^k[s,x_i^*(s,\sigma_1,\sigma_2), u_i^*[\a_{1i}(s)\ \mathcal{A}_i(s,\sigma_1,\sigma_2)],u_{-i}^*[\a_{2,-i}(s)\ \mathcal{A}_{-i}(s,\sigma_1,\sigma_2)]]ds\notag\\&\hspace{.5cm}+\omega_{i,k}^{\tau_1}\big[s,x_i^*(s,\sigma_1,\sigma_2), u_i^*[\a_{1i}(s)\ \mathcal{A}_i(s,\sigma_1,\sigma_2)],u_{-i}^*[\a_{2,-i}(s)\ \mathcal{A}_{-i}(s,\sigma_1,\sigma_2)]\big]\ dB^k(s,\sigma_1,\sigma_2),
\end{align}
with initial condition $x_i^*(0,\sigma_{1i},\sigma_{2i})=x_0^{i*}$ and
\begin{align}\label{nas11.1}
&dx_i(s,\sigma_1,\sigma_2)\notag\\&=\mu_i^{\tau_1}[s,x_i(s,\sigma_1,\sigma_2),u_i[\a_{1i}^{\hat{\rho}}(s)\ \mathcal{A}_i(s,\sigma_1,\sigma_2)],u_{-i}^*[\a_{2,-i}^{\tilde{\rho}}(s)\ \mathcal{A}_{-i}(s,\sigma_1,\sigma_2)]]ds\notag\\&\hspace{1cm}+\omega_{i,k}^{\tau_1}[s,x_i(s,\sigma_1,\sigma_2),u_i[\a_{1i}^{\hat{\rho}}(s)\ \mathcal{A}_i(s,\sigma_1,\sigma_2)],u_{-i}^*[\a_{2,-i}^{\tilde{\rho}}(s)\ \mathcal{A}_{-i}(s,\sigma_1,\sigma_2)]]\ dB^k(s,\sigma_1,\sigma_2),
\end{align}	
with initial condition $x_i(0,\sigma_{1i},\sigma_{2i})=x_0^{i}$, for $i=1,...,n$ where $\a_{1i}^{\hat{\rho}}(s)$ is the degree of cooperation of firm $i$ in its strategy  and $\a_{2,-i}^{\tilde{\rho}}(s)$ is the degree of cooperation of the other firms apart from firm $i$ in terms of firm $i$'s point of view.
\end{defn}

\medskip

 Hence, under a cooperative Nash equilibrium firm $i$'s optimization problem is to find $u_i$ such that
 \begin{align}\label{nas9}
 \max_{u_i\in U} {\Pi}^N(u_i,t)=\max_{u_i\in U}\ &\E \int_\Sigma \pi_i\bigg[s,x_i(s,\sigma_1,\sigma_2),u_i[\a_{1i}^{\hat{\rho}}(s) \mathcal{A}_i(s,\sigma_1,\sigma_2)],\notag\\&\hspace{1cm}u_{-i}^*[\a_{2,-i}^{\tilde{\rho}}(s) \mathcal{A}_{-i}(s,\sigma_1,\sigma_2)]\bigg] b_i d\sigma^3,
 \end{align}
 subject to 
\begin{align}\label{nas9.1}
& dx_i(s,\sigma_1,\sigma_2)\notag\\&=\mu_i^{\tau_1}[s,x_i(s,\sigma_1,\sigma_2),u_i[\a_{1i}^{\hat{\rho}}(s) \mathcal{A}_i(s,\sigma_1,\sigma_2)],u_{-i}^*[\a_{2,-i}^{\tilde{\rho}}(s) \mathcal{A}_{-i}(s,\sigma_1,\sigma_2)]] ds\notag\\&\hspace{.25cm}+\omega_{i,k}^{\tau_1}[s,x_i(s,\sigma_1,\sigma_2),u_i[\a_{1i}^{\hat{\rho}}(s) \mathcal{A}_i(s,\sigma_1,\sigma_2)],u_{-i}^*[\a_{2,-i}^{\tilde{\rho}}(s) \mathcal{A}_{-i}(s,\sigma_1,\sigma_2)]] dB^k(s,\sigma_1,\sigma_2),
\end{align}
 with initial condition $x_i(0,\sigma_{1i},\sigma_{2i})=x_0^i$ for all $i=1,...,n$ where $u_{-i}^*[\a_{2,-i}^{\tilde{\rho}}(s) \mathcal{A}_{-i}(s,\sigma_1,\sigma_2)]$ is the optimized strategies other than the $i^{th}$ firm. As we assume some non-zero initial condition which represents a plane for this problem is similar to a pseudo Neuman-brane. If the initial condition is zero then it is a pseudo Dirichlet-brane. The extreme volatility of market dynamics therefore  Neuman boundary condition (first order derivative of $\sigma$ with respect to $x$ is zero) is appropriate for this case.

If firm $i$'s objective is Equation (\ref{nas9}) subject to the stochastic differential equation expressed in Equation (\ref{nas9.1}) then the Liouville-like $2$-brane action function under a $11$-dimensional Clifford torus with $\tau_1,\tau_2=\{1,2,3\}^2$, and $\theta,\nu,\hat{\theta}_k=\{4,..,11\}^3$ is the coordinate system in the transverse direction with $k=1,2,3$ is 
 \begin{align}\label{liouville}
&\mathfrak{F}_{0,t}^i(x)\notag\\&=\mbox{$\frac{1}{2}$}\int_\Sigma \E_s\biggr\{\sqrt{h} \bigg[3+h^{\tau_1\tau_2} \partial_{{\tau}_1} \chi_i^\theta \partial_{\tau_2} \chi_i^\nu N_{\theta\nu} \pi_i^\Omega[s,x_i(s,\sigma_1,\sigma_2),u_i[\a_{1i}^{\hat{\rho}}(s) \mathcal{A}_i(s,\sigma_1,\sigma_2)],\notag\\&u_{-i}^*[\a_{2,-i}^{\tilde{\rho}}(s) \mathcal{A}_{-i}(s,\sigma_1,\sigma_2)]] b_i^\Omega-\mbox{$\frac{1}{3!\sqrt h}$} \varepsilon^{\hat{\tau_1}\hat{\tau_2}\hat{\tau_3}} \partial_{\hat{\tau}_1} \chi_i^{\hat{\theta}_1} \partial_{\hat{\tau}_2} \chi_i^{\hat{\theta}_2} \partial_{\hat{\tau}_3} \chi_i^{\hat{\theta}_3} H_{\hat{\theta}_1\hat{\theta}_2\hat{\theta}_3}\times\notag\\& \pi_i^{1-\Omega}[s,x_i(s,\sigma_1,\sigma_2),u_i[\a_{1i}^{\hat{\rho}}(s) \mathcal{A}_i(s,\sigma_1,\sigma_2)],u_{-i}^*[\a_{2,-i}^{\tilde{\rho}}(s) \mathcal{A}_{-i}(s,\sigma_1,\sigma_2)]]  b_i^{1-\Omega}\notag\\&-Q\widetilde{R}\overline {x}-\lambda(s+ds,\sigma_1,\sigma_2)\big[x_i(s+ds,\sigma_1,\sigma_2)-x_i(s,\sigma_1,\sigma_2)\notag\\&-\mu_i^{\tau_1}[s,x_i(s,\sigma_1,\sigma_2),u_i[\a_{1i}^{\hat{\rho}}(s) \mathcal{A}_i(s,\sigma_1,\sigma_2)],u_{-i}^*[\a_{2,-i}^{\tilde{\rho}}(s) \mathcal{A}_{-i}(s,\sigma_1,\sigma_2)]]ds\notag\\&-\omega_{i,k}^{\tau_1}[s,x_i(s,\sigma_1,\sigma_2),u_i[\a_{1i}^{\hat{\rho}}(s) \mathcal{A}_i(s,\sigma_1,\sigma_2)],u_{-i}^*[\a_{2,-i}^{\tilde{\rho}}(s) \mathcal{A}_{-i}(s,\sigma_1,\sigma_2)]] dB^k(s,\sigma_1,\sigma_2)\big]\bigg] d\sigma^{3}\biggr\},
\end{align}
where $\Sigma$ is the $3$-dimensional strategy volume in time interval $[0,t]$, $h$ is the determinant of $3$-dimensional volume metric $h_{\tau_1\tau_2}$, $\Omega\in(0,1)$ is the degrees of freedom for the profit $\pi_i$, $\varepsilon^{\hat{\tau}_1\hat{\tau}_2\hat{\tau}_3}$ is Levi-Civita symbol corresponding to the transverse dimensions of the strategy field with its antisymmetric metric defined as $H_{\hat{\theta}_1\hat{\theta}_2\hat{\theta}_3}=-\varepsilon_{\hat{\tau}_1\hat{\tau}_2\hat{\tau}_3} \mbox{det}^{-1}(g_{\tau_1\tau_2})$ with metric tensor $g_{\tau_1\tau_2}$, $Q\in [2,\infty)$ is the stubbornness measure, $\widetilde R$ is the Ricci curvature scalar and $\overline {x}$ is the mean measure of all possible market shares.
 
 \begin{prop}\label{thm1}
 If firm $i$'s objective satisfies Liouville-like $2$-brane action expressed in Equation (\ref{liouville}) in an affine non-abelian conformal gauge strategy spacetime such that Assumptions \ref{as1}-\ref{as3}, and Axiom \ref{ax1} hold then there exists a strictly monotonically increasing function $F^i$ with a unique Haar measure $de dc dx$ such that for a large positive finite number $M$ and for a small time interval $\epsilon\ra 0$ there exists a quantum Schr\"odinger-like equation (more precisely a Dirac-like equation),
 \begin{align}\label{Schrodinger}
 \partial_s \Psi_{s}^{\tau,i}(x) &=\mbox{$\frac{\iota}{2M}$}\ F_0^i\ D_{\tau_1}D^{\tau_1}\Psi_s^{\tau,i},
 \end{align}
 where $D_{\tau_1}D^{\tau_1}\Psi_s^{\tau,i}$ is covariant Laplacian of $i^{th}$ firm's transition function $\Psi_s^{\tau,i}$, and $i^{th}$ firm's optimal degree of cooperation $\hat\rho^*$ is obtained by optimizing Equation (\ref{Schrodinger}) with respect to $\hat{\rho}$ and solve for it.
 \end{prop}
 
 \begin{proof}
 Equation (\ref{nas9.1}) can be written as
 \begin{align}\label{nas9.2}
 & x_i(s+ds,\sigma_1,\sigma_2)-x_i(s,\sigma_1,\sigma_2)\notag\\&=\mu_i^{\tau_1}[s,x_i(s,\sigma_1,\sigma_2),u_i[\a_{1i}^{\hat{\rho}}(s) \mathcal{A}_i(s,\sigma_1,\sigma_2)],u_{-i}^*[\a_{2,-i}^{\tilde{\rho}}(s) \mathcal{A}_{-i}(s,\sigma_1,\sigma_2)]] ds\notag\\&+\omega_{i,k}^{\tau_1}[s,x_i(s,\sigma_1,\sigma_2),u_i[\a_{1i}^{\hat{\rho}}(s) \mathcal{A}_i(s,\sigma_1,\sigma_2)],u_{-i}^*[\a_{2,-i}^{\tilde{\rho}}(s) \mathcal{A}_{-i}(s,\sigma_1,\sigma_2)]] dB^k(s,\sigma_1,\sigma_2).
 \end{align}
 Therefore, the Liouville like $2$-brane action function \cite{mavromatos1989} of the $i^{th}$ firm in time $[0,t]$ becomes,
 \begin{align}\label{9.2}
 &\mathfrak{F}_{0,t}^i(x)\notag\\&=\mbox{$\frac{1}{2}$}\int_\Sigma \E_s\biggr\{\sqrt{h} \bigg[3+h^{\tau_1\tau_2} \partial_{\tau_1} \chi_i^\theta\ \partial_{\tau_2} \chi_i^\nu N_{\theta\nu} \pi_i^\Omega[s,x_i(s,\sigma_1,\sigma_2),u_i[\a_{1i}^{\hat{\rho}}(s) \mathcal{A}_i(s,\sigma_1,\sigma_2)],\notag\\&u_{-i}^*[\a_{2,-i}^{\tilde{\rho}}(s) \mathcal{A}_{-i}(s,\sigma_1,\sigma_2)]] b_i^\Omega-\mbox{$\frac{1}{3!\sqrt h}$} \chi_i^{\hat{\theta}_1} \partial_{\hat{\tau}_2} \chi_i^{\hat{\theta}_2} \partial_{\hat{\tau}_3} \chi_i^{\hat{\theta}_3} H_{\hat{\theta}_1\hat{\theta}_2\hat{\theta}_3}\times\notag\\& \pi_i^{1-\Omega}[s,x_i(s,\sigma_1,\sigma_2),u_i[\a_{1i}^{\hat{\rho}}(s) \mathcal{A}_i(s,\sigma_1,\sigma_2)],u_{-i}^*[\a_{2,-i}^{\tilde{\rho}}(s) \mathcal{A}_{-i}(s,\sigma_1,\sigma_2)]] b_i^{1-\Omega}\notag\\&-Q\widetilde{R}\bar{x}-\lambda(s+ds,\sigma_1,\sigma_2)\big[x_i(s+ds,\sigma_1,\sigma_2)-x_i(s,\sigma_1,\sigma_2)\notag\\&-\mu_i^{\tau_1}[s,x_i(s,\sigma_1,\sigma_2),u_i[\a_{1i}^{\hat{\rho}}(s) \mathcal{A}_i(s,\sigma_1,\sigma_2)],u_{-i}^*[\a_{2,-i}^{\tilde{\rho}}(s) \mathcal{A}_{-i}(s,\sigma_1,\sigma_2)]] ds\notag\\&-\omega_{i,k}^{\tau_1}[s,x_i(s,\sigma_1,\sigma_2),u_i[\a_{1i}^{\hat{\rho}}(s) \mathcal{A}_i(s,\sigma_1,\sigma_2)],u_{-i}^*[\a_{2,-i}^{\tilde{\rho}}(s) \mathcal{A}_{-i}(s,\sigma_1,\sigma_2)]] dB^k(s,\sigma_1,\sigma_2)\big]\bigg] d\sigma^{3}\biggr\},
 \end{align}
 where $\widetilde{R}$ is Ricci curvature scalar, $\chi_i^\theta$ is $i^{th}$ firm's embedded $\theta$-coordinates with $N_{\theta\nu}$ as the background $11$-dimensional random metric with its quantum strategy volume $\int_{\Sigma} \sqrt{h} \exp(\gamma b_i) d\sigma^3$,  $\partial_{\tau_1} \chi_i^\theta=\partial \chi_i^\theta /\partial\sigma^{\tau_1}$, $\partial_{\tau_2} \chi_i^\nu=\partial \chi_i^\nu /\partial\sigma^{\tau_2}$,  $\partial_{\hat{\tau}_1} \chi_i^{\hat{\theta}_1}=\partial \chi_i^{\hat{\theta}_1} /\partial\sigma^{\hat{\tau}_1}$, $\partial_{\hat{\tau}_2} \chi_i^{\hat{\theta}_2}=\partial \chi_i^{\hat{\theta}_2} /\partial\sigma^{\hat{\tau}_2}$ and $\partial_{\hat{\tau}_3} \chi_i^{\hat{\theta}_3}=\partial \chi_i^{\hat{\theta}_3} /\partial\sigma^{\hat{\tau}_3}$, where $\theta,\nu,\hat{\theta}_k=\{4,..,11\}^3$ is the coordinate system in the transverse direction with $k=1,2,3$ and coupled with antisymmetric metric $H_{\hat{\theta}_1\hat{\theta}_2\hat{\theta}_3}=-\varepsilon_{\hat{\tau}_1\hat{\tau}_2\hat{\tau}_3} \mbox{det}^{-1}(g_{\tau_1\tau_2})$ with metric tensor $g_{\tau_1\tau_2}$ and $\varepsilon^{\hat{\tau}_1\hat{\tau}_2\hat{\tau}_3}$ is the contravariant Levi-Civita symbol corresponding to transverse dimensions \cite{kaku2012}.  In Equation (\ref{9.2}) $\Omega\in(0,1)$ represents the degree of freedom of $i^{th}$ firm's profit function $\pi_i$ and the stubbornness $b_i$ in the world volume and transverse directions respectively. Furthermore, first two terms in the parenthesis in Equation (\ref{9.2}) are a generalization of the Nambu-Goto like action function \cite{kaku2012}, third term is asymmetric part occurred due to transverse dimensions and we subtracted to get a finite measure of the action. 
 
 The transverse direction for the fourth dimension represents the parallel strategy spaces with the same coordinate systems or $\theta=\hat{\theta}_k$, the fifth dimension is firm $i$ is traveling through time in those parallel strategy spaces, the sixth dimension is parallel strategy spaces with different coordinate systems or $\theta\neq\hat{\theta}_k$, the seventh dimension is that firm is traveling through time in those parallel strategy spaces with $\theta\neq\hat{\theta}_k$, the eighth dimension is strategy spaces with different initial conditions and same economic laws, the ninth dimension is different initial conditions and different economic laws, the tenth dimension is same initial conditions with different economic laws, and the eleventh dimension is the case where all know economic laws fail.  
 
 In Equation (\ref{9.2}), $Q$ is the stubbornness measure on the spacetime and defined as $Q=2/\gamma+\gamma/2$ to have some conformal invariance \cite{gwynne2016}.  If the governing body of the firm $i$ becomes too stubborn, then $Q=2$ and on the other hand, a liberal governing body leads an enormous value of $Q$ (i.e $Q\ra\infty$). As perfect flexibility is unrealistic, we can ignore this part so that $Q\in[2,\infty)$. Equation (\ref{9.2}) represents a quantum Lagrangian-like action in strategy field where first two terms on the right hand side parenthesis correspond to the kinetic energy in world volume and transverse direction respectively, the term multiplied by a strategy constant $\lambda$ is the potential energy and $Q\widetilde{R}\bar {x}$ is the potential energy of the strategy spacetime where $\bar {x}$ is mean of all possible measures of market shares and  $\partial_{\tau_1}\chi_i^\theta\partial_{\tau_2}\chi_i^\nu N_{\theta\nu}$ is the pull-back of the background field $U$ \cite{kimura2016}.
 
We assume that the background strategy field $U$ is dynamic, and therefore conformal symmetry is a gauge symmetry similar to diffeomorphism invariance and Weyl invariance. For the brane metric $h_{\tau_1\tau_2}(\sigma)$ with parameter vector $\sigma$ (in our case $\sigma=\{s,\sigma_1,\sigma_2\}$) has the gauge transformation
$$h_{\tau_1\tau_2}(\sigma)\ra e^{2\omega(\sigma)}\ \frac{\partial\sigma^{\tau_3}}{\partial\sigma^{'\tau_1}}\frac{\partial\sigma^{\tau_4}}{\partial\sigma^{'\tau_2}}\ h_{\tau_3\tau_4}(\sigma).$$ 
In two dimensions these gauge symmetries allow us to put the metric into any form we like. This is true locally, but not globally. In other words, it remains true if the world volume has the topology of a cylinder or a sphere but not for the surface of higher dimensions. Furthermore, fixing $h$ locally fails to fix all gauge symmetries and also conformal symmetries. Therefore, we introduce the Faddeev-Popov ghost action function with $\mathfrak{F}_{0,t}^i(x)$ (\cite{faddeev1967}, \cite{t200}). Another reason to introduce the ghost field is to cancel unphysical gauge degrees of freedom. Suppose, there exist  Grassmann-valued anti-commuting fields (ghost fields) with tensors $e^{\tau_1\tau_2}$ and $c^{\tau_1}$. Then the Faddeev-Popov determinant of $h$ in the case of $i^{th}$ firm is 
\begin{align}\label{9.20}
\Delta_{FP}[h]=\int_{\Sigma} \exp\left\{\iota\epsilon\ \mathfrak{F}_{0,t}^{iG}(x)\right\} de dc,
\end{align}
 where the ghost action is defined as 
 \begin{align}\label{9.21}
 \mathfrak{F}_{0,t}^{iG}(e,c,h)=\frac{1}{2\pi\epsilon}\int_{\Sigma} \sqrt{h} e^{\tau_1\tau_2} \nabla^{\tau_1}c^{\tau_2} d\sigma^{3},
 \end{align}
 and $\nabla^{\tau_1}$ is the contravariant derivative with respect to $\tau_1$.
 
 Defining $\Delta s=\epsilon>0$ and for some positive constant $L_\epsilon$ with $\epsilon\ra 0$ transition function for a small time interval $[s,s+\epsilon]$ of firm $i$ can be written as 
 \begin{align}\label{9.3}
 \Psi_{s,s+\epsilon}^i(x)&=\frac{1}{L_\epsilon} \int_{ \widetilde{\Sigma}_s} \exp\bigg\{\iota\epsilon  \left[ \mathfrak{F}_{s,s+\epsilon}^i(x)+\mathfrak{F}_{0,t}^{iG}(e,c,h)\right]\bigg\}\times\notag\\& \Psi_s^i(x)\ de(s,\sigma_1,\sigma_2)\ dc(s,\sigma_1,\sigma_2)\ dx(s,\sigma_1,\sigma_2),
 \end{align}
 where $\widetilde{\Sigma}_s$ is the collection of all possible fields, and $\mathfrak{F}_{s,s+\epsilon}(x)$ is a local Liouville like action function in $[s,s+\epsilon]$ such that,
 \begin{align}\label{9.4}
 \Psi_{0,t}^i(x)&=\frac{1}{(L_\epsilon)^l} \int_{ (\widetilde{\Sigma}_s)^l} \exp\bigg\{\iota \epsilon  \sum_{j=1}^l\left[\mathfrak{F}_{s_j,s_{j+1}}^i(x)+\mathfrak{F}_{s_j,s_{j+1}}^{iG}(e,c,h)\right]\bigg\}\times\notag\\&  \Psi_s^i(x)\prod_{j=1}^l de(s_j,\sigma_1,\sigma_2) \prod_{j=1}^ldc(s_j,\sigma_1,\sigma_2)\prod_{j=1}^l dx(s_j,\sigma_1,\sigma_2),
 \end{align}
 where the time interval $[0,t]$ has been divided into $l> 1$ number of equal length subintervals.
 
 After using Fubini's theorem Equation (\ref{9.2}) becomes,
 \begin{align}
 &\mathfrak{F}_{0,t}^i(x)+\mathfrak{F}_{0,t}^{iG}(e,c,h)\notag\\&=\mbox{$\frac{1}{2}$} \E_s\biggr\{\int_\Sigma \sqrt{h} \bigg[3+h^{\tau_1\tau_2} \partial_{\tau_1} \chi_i^\theta \partial_{\tau_2} \chi_i^\nu\ N_{\theta\nu} \pi_i^\Omega[s,x_i(s,\sigma_1,\sigma_2),u_i[\a_{1i}^{\hat{\rho}}(s) \mathcal{A}_i(s,\sigma_1,\sigma_2)],\notag\\&u_{-i}^*[\a_{2,-i}^{\tilde{\rho}}(s) \mathcal{A}_{-i}(s,\sigma_1,\sigma_2)]]  b_i^\Omega-\mbox{$\frac{1}{3!\sqrt h}$} \varepsilon^{\hat{\tau_1}\hat{\tau_2}\hat{\tau_3}}\ \partial_{\hat{\tau}_1} \chi_i^{\hat{\theta}_1} \partial_{\hat{\tau}_2} \chi_i^{\hat{\theta}_2} \partial_{\hat{\tau}_3} \chi_i^{\hat{\theta}_3} H_{\hat{\theta}_1\hat{\theta}_2\hat{\theta}_3}\times\notag\\& \pi_i^{1-\Omega}[s,x_i(s,\sigma_1,\sigma_2),u_i[\a_{1i}^{\hat{\rho}}(s) \mathcal{A}_i(s,\sigma_1,\sigma_2)],u_{-i}^*[\a_{2,-i}^{\tilde{\rho}}(s) \mathcal{A}_{-i}(s,\sigma_1,\sigma_2)]]  b_i^{1-\Omega}+\frac{e^{\tau_1\tau_2}\nabla^{\tau_1}c^{\tau_2}}{\pi\epsilon}\notag\\&-Q\widetilde{R}\overline {x}-\lambda_i(s+ds,\sigma_1,\sigma_2)\big[x_i(s+ds,\sigma_1,\sigma_2)-x_i(s,\sigma_1,\sigma_2)\notag\\&-\mu_i^{\tau_1}[s,x_i(s,\sigma_1,\sigma_2),u_i[\a_{1i}^{\hat{\rho}}(s) \mathcal{A}_i(s,\sigma_1,\sigma_2)],u_{-i}^*[\a_{2,-i}^{\tilde{\rho}}(s) \mathcal{A}_{-i}(s,\sigma_1,\sigma_2)]] ds\notag\\&-\omega_{i,k}^{\tau_1}[s,x_i(s,\sigma_1,\sigma_2),u_i[\a_{1i}^{\hat{\rho}}(s) \mathcal{A}_i(s,\sigma_1,\sigma_2)],u_{-i}^*[\a_{2,-i}^{\tilde{\rho}}(s) \mathcal{A}_{-i}(s,\sigma_1,\sigma_2)]]\ dB^k(s,\sigma_1,\sigma_2)\big]\bigg] d\sigma^{3}\biggr\},\notag
 \end{align}
 and
 \begin{align}\label{9.6}
 &\mathfrak{F}_{s,s+\epsilon}^i(x)+\mathfrak{F}_{s,s+\epsilon}^{iG}(e,c,h)\notag\\&=\mbox{$\frac{1}{2}$}\E_s\biggr\{\int_\Sigma\ \sqrt{h} \bigg[3+h^{\tau_1\tau_2} \partial_{\tau_1} \chi_i^\theta \partial_{\tau_2} \chi_i^\nu\ N_{\theta\nu} \pi_i^{\Omega}[\nu,x_i(\nu,\sigma_1,\sigma_2),u_i[\a_{1i}^{\hat{\rho}}(\nu) \mathcal{A}_i(\nu,\sigma_1,\sigma_2)],\notag\\&u_{-i}^*[\a_{2,-i}^{\tilde{\rho}}(\nu) \mathcal{A}_{-i}(\nu,\sigma_1,\sigma_2)]]  b_i^\Omega-\mbox{$\frac{1}{3!\sqrt h}$} \varepsilon^{\hat{\tau_1}\hat{\tau_2}\hat{\tau_3}}\partial_{\hat{\tau}_1} \chi_i^{\hat{\theta}_1} \partial_{\hat{\tau}_2} \chi_i^{\hat{\theta}_2} \partial_{\hat{\tau}_3} \chi_i^{\hat{\theta}_3} H_{\hat{\theta}_1\hat{\theta}_2\hat{\theta}_3}\times\notag\\& \pi_i^{1-\Omega}[\nu,x_i(s,\sigma_1,\sigma_2),u_i[\a_{1i}^{\hat{\rho}}(\nu) \mathcal{A}_i(\nu,\sigma_1,\sigma_2)],u_{-i}^*[\a_{2,-i}^{\tilde{\rho}}(\nu) \mathcal{A}_{-i}(\nu,\sigma_1,\sigma_2)]]  b_i^{1-\Omega}+\frac{e^{\tau_1\tau_2} \nabla^{\tau_1}c^{\tau_2}}{\pi\epsilon}\notag\\&-Q\widetilde{R}\overline {x}-\lambda_i(\nu+d\nu,\sigma_1,\sigma_2)\big[x_i(\nu+d\nu,\sigma_1,\sigma_2)-x_i(\nu,\sigma_1,\sigma_2)\notag\\&-\mu_i^k[\nu,x_i(\nu,\sigma_1,\sigma_2),u_i[\a_{1i}^{\hat{\rho}}(\nu) \mathcal{A}_i(\nu,\sigma_1,\sigma_2)],u_{-i}^*[\a_{2,-i}^{\tilde{\rho}}(\nu) \mathcal{A}_{-i}(\nu,\sigma_1,\sigma_2)]] d\nu\notag\\&-\omega_{i,k}^{\tau_1}[\nu,x_i(\nu,\sigma_1,\sigma_2),u_i[\a_{1i}^{\hat{\rho}}(\nu) \mathcal{A}_i(\nu,\sigma_1,\sigma_2)],u_{-i}^*[\a_{2,-i}^{\tilde{\rho}}(\nu) \mathcal{A}_{-i}(\nu,\sigma_1,\sigma_2)]] dB^k(\nu,\sigma_1,\sigma_2)\big]\bigg] d\sigma^{3}\biggr\}.
 \end{align}
 
 \medskip
 
 In Equation (\ref{9.6}) $\Delta x_i(\nu,\sigma_1,\sigma_2)=x_i(\nu+d\nu,\sigma_1,\sigma_2)-x_i(\nu,\sigma_1,\sigma_2)$ is an It\^o's process. Therefore, by \cite{ito1962} there exists a $3$-dimensional tensor attached to $x_0^i$ for firm $i$ obtained by Levi-Civita parallel displacement, $_ig_{k_1,...,k_p}^{\tau_1}[\nu,x(\nu,\sigma_1,\sigma_2)]\in C^2([0,\infty)\times\mathbb{R}^{3})$ which satisfies Assumptions \ref{as1}-\ref{as3} and $_iY_{k_1,...,k_p}^{\tau_1}(\nu)=\ _ig_{k_1,...,k_p}^{\tau_1}[\nu,x(\nu,\sigma_1,\sigma_2)]$ where $_iY_{k_1,...,k_p}^{\tau_1}(\nu)$ is $i^{th}$ firm's ${\tau_1}^{th}$-dimensional It\^o's process.
 
 \medskip
 
After assuming 
\begin{align}
& _ig_{k_1,...,k_p}^{\tau_1}[\nu+\Delta\nu,x(\nu,\sigma_1,\sigma_2)+\Delta x(\nu,\sigma_1,\sigma_2) ]\notag\\&\approx\lambda_i(\nu+d\nu)\big[\Delta x_i(\nu,\sigma_1,\sigma_2)-\mu_i^{\tau_1}[\nu,x_i(\nu,\sigma_1,\sigma_2),u_i[\a_{1i}^{\hat{\rho}}(\nu) \mathcal{A}_i(\nu,\sigma_1,\sigma_2)],u_{-i}^*[\a_{2,-i}^{\tilde{\rho}}(\nu) \mathcal{A}_{-i}(\nu,\sigma_1,\sigma_2)]] d\nu\notag\\&-\omega_{i,k}^{\tau_1}[\nu,x_i(\nu,\sigma_1,\sigma_2),u_i[\a_{1i}^{\hat{\rho}}(\nu)\mathcal{A}_i(\nu,\sigma_1,\sigma_2)],u_{-i}^*[\a_{2,-i}^{\tilde{\rho}}(\nu) \mathcal{A}_{-i}(\nu,\sigma_1,\sigma_2)]] dB^k(\nu,\sigma_1,\sigma_2)\big],\notag
\end{align} 
Equation (\ref{9.6}) becomes
\begin{align}\label{9.7}
 &\mathfrak{F}_{s,s+\epsilon}^i(x)+\mathfrak{F}_{s,s+\epsilon}^{iG}(e,c,h)\notag\\&=\mbox{$\frac{1}{2}$} \E_s \biggr\{\int_\Sigma \sqrt{h} \bigg[3+h^{\tau_1\tau_2} \partial_{\tau_1} \chi_i^\theta\ \partial_{\tau_2} \chi_i^\nu\ N_{\theta\nu} \pi_i^\Omega[\nu,x_i(\nu,\sigma_1,\sigma_2),u_i[\a_{1i}^{\hat{\rho}}(\nu)\ \mathcal{A}_i(\nu,\sigma_1,\sigma_2)],\notag\\&u_{-i}^*[\a_{2,-i}^{\tilde{\rho}}(\nu) \mathcal{A}_{-i}(\nu,\sigma_1,\sigma_2)]] b_i^\Omega-\mbox{$\frac{1}{3!\sqrt h}$} \varepsilon^{\hat{\tau_1}\hat{\tau_2}\hat{\tau_3}} \partial_{\hat{\tau}_1} \chi_i^{\hat{\theta}_1} \partial_{\hat{\tau}_2} \chi_i^{\hat{\theta}_2} \partial_{\hat{\tau}_3} \chi_i^{\hat{\theta}_3} H_{\hat{\theta}_1\hat{\theta}_2\hat{\theta}_3}\times\notag\\& \pi_i^{1-\Omega}[\nu,x_i(\nu,\sigma_1,\sigma_2),u_i[\a_{1i}^{\hat{\rho}}(\nu) \mathcal{A}_i(\nu,\sigma_1,\sigma_2)],u_{-i}^*[\a_{2,-i}^{\tilde{\rho}}(\nu) \mathcal{A}_{-i}(\nu,\sigma_1,\sigma_2)]] \ b_i^{1-\Omega}+\frac{e^{\tau_1\tau_2} \nabla^{\tau_1}c^{\tau_2}}{\pi\epsilon}\notag\\&-Q\widetilde{R}\bar{x}-\ _ig_{k_1,...,k_p}^{\tau_1}[\nu+\Delta\nu,x(\nu,\sigma_1,\sigma_2)+\Delta x(\nu,\sigma_1,\sigma_2) ]\bigg] d\sigma^{3}\biggr\}.
\end{align}

\medskip

Following \cite{ito1950}, \cite{ito1962} and \cite{yeung2006}, the generalized It\^o's formula of a differentiable manifold yields
\begin{align}\label{9.8}
& _ig_{k_1,..,k_{v-1},k,k_{v+1},...,k_p}^{\tau_1}[\nu+\Delta\nu,x(\nu,\sigma_1,\sigma_2)+\Delta x(\nu,\sigma_1,\sigma_2) ]\notag\\&=\biggr[\ _ig_{k_1,..,k_{v-1},k,k_{v+1},...,k_p}^{\tau_1}+\  _ig_{k_1,..,k_{v-1},k,k_{v+1},...,k_p;\ \nu}^{\tau_1}\notag\\&+\mbox{$\frac{1}{2}$}\sum_{v=1}^{2}\ _ih^{\tau_1\tau_2}\bigg(\frac{\partial\ _i\Gamma_{\tau_2k_v}^k }{\partial x^{\tau_1}}+\ _i\Gamma_{\tau_2k_v}^{\tilde{k}}\ _i\Gamma_{\tau_1\tilde{k}}^k\bigg)\otimes\ _ig_{k_1,..,k_{v-1},k,k_{v+1},...,k_p}^{\tau_1}\notag\\&+\mbox{$\frac{1}{2}$}\sum_{v=1}^{2}\sum_{w=1}^{2}\ _ih^{\tau_1\tau_2}\ _i\Gamma_{\tau_1k_v}^k\ _i\Gamma_{\tau_2k_w}^{\tilde{k}}\ \otimes\ _ig_{k_1,..,k_{v-1},k,k_{v+1},...,k_{w-1},\tilde{k},k_{w+1},...,k_p}^{\tau_1}+o(1) \biggr] d\nu\notag\\&+\sum_{v=1}^{2}\ _i\Gamma_{jk_v}^k\ _ig_{k_1,..,k_{v-1},k,k_{v+1},...,k_p}^{\tau_1} dB^j,
\end{align}
with the initial condition $_ig_{k_1,..,k_{v-1},k,k_{v+1},...,k_p}^{\tau_1}(0,\sigma_{1i},\sigma_{2i})=\ _ig_{k_1,..,k_{v-1},k,k_{v+1},...,k_p}^{\tau_10}$, where\\ $_ig_{k_1,..,k_{v-1},k,k_{v+1},...,k_p; \nu}^{\tau_1}=\frac{\partial}{\partial x^\nu}\ _ig_{k_1,..,k_{v-1},k,k_{v+1},...,k_p}^{\tau_1}$ and $dB^{\tau_1}(\nu,\sigma_1,\sigma_2)\ dB^{\tau_2}(\nu,\sigma_1,\sigma_2)\sim h^{\tau_1\tau_2}\ d\nu$.

In Equation (\ref{9.7}), the determinant of the world volume (brane) $\sqrt{h}$ can be taken out of the integral as a constant because it is a fixed value of a field. Now for a small interval around $s$ with $\epsilon\ra 0$, using result in Equation (\ref{9.8}) into (\ref{9.7}), dividing through $\epsilon$ and taking conditional expectation we get,
\begin{align}\label{9.9}
&\mathfrak{F}_{s,s+\epsilon}^i(x)+\mathfrak{F}_{s,s+\epsilon}^{iG}(e,c,h)\notag\\&=\mbox{$\frac{\sqrt{h}}{2}$} \bigg[3+h^{\tau_1\tau_2} \partial_{\tau_1} \chi_i^\theta \partial_{\tau_2} \chi_i^\nu N_{\theta\nu} \pi_i^\Omega[\nu,x_i(s,\sigma_1,\sigma_2),u_i[\a_{1i}^{\hat{\rho}}(s) \mathcal{A}_i(s,\sigma_1,\sigma_2)],\notag\\&u_{-i}^*[\a_{2,-i}^{\tilde{\rho}}(s) \mathcal{A}_{-i}(s,\sigma_1,\sigma_2)]] b_i^\Omega-\mbox{$\frac{1}{3!\sqrt h}$} \varepsilon^{\hat{\tau_1}\hat{\tau_2}\hat{\tau_3}} \partial_{\hat{\tau}_1} \chi_i^{\hat{\theta}_1} \partial_{\hat{\tau}_2} \chi_i^{\hat{\theta}_2} \partial_{\hat{\tau}_3} \chi_i^{\hat{\theta}_3} H_{\hat{\theta}_1\hat{\theta}_2\hat{\theta}_3}\times\notag\\& \pi_i^{1-\Omega}[s,x_i(s,\sigma_1,\sigma_2),u_i[\a_{1i}^{\hat{\rho}}(s) \mathcal{A}_i(s,\sigma_1,\sigma_2)],u_{-i}^*[\a_{2,-i}^{\tilde{\rho}}(s) \mathcal{A}_{-i}(s,\sigma_1,\sigma_2)]]  b_i^{1-\Omega}\notag\\&+\frac{e^{\tau_1\tau_2} \nabla^{\tau_1}c^{\tau_2}}{\pi\epsilon}-Q\widetilde{R}\overline {x}-\ _ig_{k_1,..,k_{v-1},k,k_{v+1},...,k_p}^{\tau_1}-\  _ig_{k_1,..,k_{v-1},k,k_{v+1},...,k_p;\ \nu}^{\tau_1}\notag\\&-\mbox{$\frac{1}{2}$}\sum_{v=1}^{2}\ _ih^{\tau_1\tau_2}\bigg(\frac{\partial\ _i\Gamma_{\tau_2k_v}^k }{\partial x^{\tau_1}}+\ _i\Gamma_{\tau_2k_v}^{\tilde{k}}\ _i\Gamma_{\tau_1\tilde{k}}^k\bigg)\otimes\ _ig_{k_1,..,k_{v-1},k,k_{v+1},...,k_p}^{\tau_1}\notag\\&-\mbox{$\frac{1}{2}$}\sum_{v=1}^{2}\sum_{w=1}^{2}\ _ih^{\tau_1\tau_2}\ _i\Gamma_{\tau_1k_v}^k\ _i\Gamma_{\tau_2k_w}^{\tilde{k}}\ \otimes\ _ig_{k_1,..,k_{v-1},k,k_{v+1},...,k_{w-1},\tilde{k},k_{w+1},...,k_p}^{\tau_1}+o(1) \bigg],
\end{align}
where $\E_s[\Delta B^j(s,\sigma_1,\sigma_2)]=0$ and $\E_s[o(\epsilon)]/\epsilon=0$ as $\epsilon\ra 0$.

\medskip

Define the pull-back tensor of world volume as $\widehat{N}_{\tau_1\tau_2}=\partial_{\tau_1} \chi_i^\theta\ \partial_{\tau_2} \chi_i^\nu\ N_{\theta\nu}$ and towards the transverse direction as $\widehat{H}_{\hat{\tau}_1\hat{\tau}_2\hat{\tau}_3}= \partial_{\hat{\tau}_1} \chi_i^{\hat{\theta}_1} \partial_{\hat{\tau}_2} \chi_i^{\hat{\theta}_2} \partial_{\hat{\tau}_3} \chi_i^{\hat{\theta}_3}\ H_{\hat{\theta}_1\hat{\theta}_2\hat{\theta}_3}$. The expression in Equation (\ref{9.9}) then becomes,
\begin{align}\label{9.11}
&\mathfrak{F}_{s,s+\epsilon}^i(x)+\mathfrak{F}_{s,s+\epsilon}^{iG}(e,c,h)\notag\\&=\mbox{$\frac{\sqrt{h}}{2}$} \bigg[3+h^{\tau_1\tau_2} \widehat{N}_{\tau_1\tau_2}(\pi_i b_i)^\Omega-\mbox{${\frac{1}{3!\sqrt h}}$} \varepsilon^{\hat{\tau}_1\hat{\tau}_2\hat{\tau}_3} \widehat{H}_{\hat{\tau}_1\hat{\tau}_2\hat{\tau}_3} (\pi_i b_i)^{1-\Omega}+\frac{e^{\tau_1\tau_2} \nabla^{\tau_1}c^{\tau_2}}{\pi\epsilon}\notag\\&\hspace{.25cm}-Q\widetilde{R}\bar{x}-\ _ig_{k_1,..,k_{v-1},k,k_{v+1},...,k_p}^{\tau_1}-\  _ig_{k_1,..,k_{v-1},k,k_{v+1},...,k_p;\ \nu}^{\tau_1}\notag\\&-\mbox{$\frac{1}{2}$}\sum_{v=1}^{2}\ _ih^{\tau_1\tau_2}\bigg(\frac{\partial\ _i\Gamma_{\tau_2k_v}^k }{\partial x^{\tau_1}}+\ _i\Gamma_{\tau_2k_v}^{\tilde{k}}\ _i\Gamma_{\tau_1\tilde{k}}^k\bigg)\otimes\ _ig_{k_1,..,k_{v-1},k,k_{v+1},...,k_p}^{\tau_1}\notag\\&-\mbox{$\frac{1}{2}$}\sum_{v=1}^{2}\sum_{w=1}^{2}\ _ih^{\tau_1\tau_2}\ _i\Gamma_{\tau_1k_v}^k\ _i\Gamma_{\tau_2k_w}^{\tilde{k}}\ \otimes\ _ig_{k_1,..,k_{v-1},k,k_{v+1},...,k_{w-1},\tilde{k},k_{w+1},...,k_p}^{\tau_1}+o(1) \bigg].
\end{align}

\medskip

Consider the transition function at the initial time $s=0$ and state $x_0$ is $\Psi_0(x_0)$. For $\epsilon\ra 0$ the transition function for $[s,s+\epsilon]$ in Equation (\ref{9.3}) becomes,
\begin{align}\label{9.12}
\Psi_{s,s+\epsilon}^i(x)&=\frac{1}{L_\epsilon} \int_{ \widetilde{\Sigma}_s} \exp\biggr\{\iota\epsilon  \bigg[\mbox{$\frac{\sqrt{h}}{2}$} \bigg(3+h^{\tau_1\tau_2} \widehat{N}_{\tau_1\tau_2} (\pi_i\ b_i)^\Omega-\mbox{${\frac{1}{3!\sqrt h}}$} \varepsilon^{\hat{\tau}_1\hat{\tau}_2\hat{\tau}_3} \widehat{H}_{\hat{\tau}_1\hat{\tau}_2\hat{\tau}_3} (\pi_i b_i)^{1-\Omega}\notag\\&+\frac{e^{\tau_1\tau_2} \nabla^{\tau_1}c^{\tau_2}}{\pi\epsilon}-Q\widetilde{R}\bar {x}-\ _ig_{k_1,..,k_{v-1},k,k_{v+1},...,k_p}^{\tau_1}-\  _ig_{k_1,..,k_{v-1},k,k_{v+1},...,k_p;\ \nu}^{\tau_1}\notag\\&-\mbox{$\frac{1}{2}$}\sum_{v=1}^{2}\ _ih^{\tau_1\tau_2}\bigg(\frac{\partial\ _i\Gamma_{\tau_2k_v}^k }{\partial x^{\tau_1}}+\ _i\Gamma_{\tau_2k_v}^{\tilde{k}}\ _i\Gamma_{\tau_1\tilde{k}}^k\bigg)\otimes\ _ig_{k_1,..,k_{v-1},k,k_{v+1},...,k_p}^{\tau_1}\notag\\&-\mbox{$\frac{1}{2}$}\sum_{v=1}^{2}\sum_{w=1}^{2}\ _ih^{\tau_1\tau_2}\ _i\Gamma_{\tau_1k_v}^k\ _i\Gamma_{\tau_2k_w}^{\tilde{k}}\ \otimes\ _ig_{k_1,..,k_{v-1},k,k_{v+1},...,k_{w-1},\tilde{k},k_{w+1},...,k_p}^{\tau_1}\bigg)  \bigg]\biggr\}\times\notag\\& \Psi_s^i(x) de(s,\sigma_1,\sigma_2) dc(s,\sigma_1,\sigma_2) dx(s,\sigma_1,\sigma_2)+o\left(\epsilon^{\frac{1}{2}}\right).
\end{align}

For $\epsilon\ra0$, keeping $\sigma_1$ and $\sigma_2$ variables constant, and defining $s+\epsilon=\tau$, a first order Taylor series expansion with respect to time on the left hand side of the Equation (\ref{9.12}) gives us
\begin{align}\label{9.13}
&\Psi_{s}^{\tau,i}(x)+\epsilon \frac{\partial \Psi_{s}^{\tau,i}(x)}{\partial s}+o(\epsilon)\notag\\& =\frac{1}{L_\epsilon} \int_{ \widetilde{\Sigma}_s} \exp\biggr\{\iota\epsilon  \bigg[\mbox{$\frac{\sqrt{h}}{2}$} \bigg(3+h^{\tau_1\tau_2} \widehat{N}_{\tau_1\tau_2} (\pi_i\ b_i)^\Omega-\mbox{${\frac{1}{3!\sqrt h}}$} \varepsilon^{\hat{\tau}_1\hat{\tau}_2\hat{\tau}_3} \widehat{H}_{\hat{\tau}_1\hat{\tau}_2\hat{\tau}_3} (\pi_i b_i)^{1-\Omega}\notag\\&+\frac{e^{\tau_1\tau_2} \nabla^{\tau_1}c^{\tau_2}}{\pi\epsilon}-Q\widetilde{R}\bar {x}-\ _ig_{k_1,..,k_{v-1},k,k_{v+1},...,k_p}^{\tau_1}-\  _ig_{k_1,..,k_{v-1},k,k_{v+1},...,k_p;\ \nu}^{\tau_1}\notag\\&-\mbox{$\frac{1}{2}$}\sum_{v=1}^{2}\ _ih^{\tau_1\tau_2}\bigg(\frac{\partial\ _i\Gamma_{\tau_2k_v}^k }{\partial x^{\tau_1}}+\ _i\Gamma_{\tau_2k_v}^{\tilde{k}}\ _i\Gamma_{\tau_1\tilde{k}}^k\bigg)\otimes\ _ig_{k_1,..,k_{v-1},k,k_{v+1},...,k_p}^{\tau_1}\notag\\&-\mbox{$\frac{1}{2}$}\sum_{v=1}^{2}\sum_{w=1}^{2}\ _ih^{\tau_1\tau_2}\ _i\Gamma_{\tau_1k_v}^k\ _i\Gamma_{\tau_2k_w}^{\tilde{k}}\ \otimes\ _ig_{k_1,..,k_{v-1},k,k_{v+1},...,k_{w-1},\tilde{k},k_{w+1},...,k_p}^{\tau_1}\bigg)  \bigg]\biggr\}\times\notag\\& \Psi_s^i(x) de dc dx+o\left(\epsilon^{\frac{1}{2}}\right).
\end{align}

For some fixed $[s,\tau]$ there exists a positive finite number $\tilde\xi$ such that $x(s,\sigma_1,\sigma_2)=x(\tau,\sigma_1,\sigma_2)+\tilde\xi$. If $\tilde\xi$ is not around zero, then for some finite positive number $\eta$ and $x>0$  we assume that $|\tilde\xi|\leq \sqrt{\eta\epsilon/x}$ such that as $\epsilon\ra 0$, $\tilde\xi$ will be a very small number. A Taylor series expansion of  Equation (\ref{9.13}) gives,
\begin{align}\label{9.14}
&\Psi_{s}^{\tau,i}(x)+\epsilon \frac{\partial \Psi_{s}^{\tau,i}(x)}{\partial s}+o(\epsilon)\notag\\& =\frac{1}{L_\epsilon} \frac{\sqrt{h}}{2}\int_{ \widetilde{\Sigma}_s} \left[\Psi_s^{\tau,i}(x)+\Delta x^{\tau_1} \partial_{\tau_1} \Psi_s^{\tau,i}(x) + \mbox{$\frac{1}{2}$} \Delta x^{\tau_1} \Delta x^{\tau_2} \partial_{\tau_1}\partial_{\tau_2}\Psi_s^{\tau,i}(x)+ o\left(\epsilon^{\frac{1}{2}}\right)\right]\notag\\&\exp\biggr\{\iota\epsilon  \bigg[3+ h^{\tau_1\tau_2} \widehat{N}_{\tau_1\tau_2} (\pi_i\ b_i)^\Omega-\mbox{${\frac{1}{3!\sqrt h}}$} \varepsilon^{\hat{\tau}_1\hat{\tau}_2\hat{\tau}_3} \widehat{H}_{\hat{\tau}_1\hat{\tau}_2\hat{\tau}_3} (\pi_i b_i)^{1-\Omega}\notag\\&+\frac{e^{\tau_1\tau_2} \nabla^{\tau_1}c^{\tau_2}}{\pi\epsilon}-Q\widetilde{R}\bar {x}-\ _ig_{k_1,..,k_{v-1},k,k_{v+1},...,k_p}^{\tau_1}-\  _ig_{k_1,..,k_{v-1},k,k_{v+1},...,k_p;\ \nu}^{\tau_1}\notag\\&-\mbox{$\frac{1}{2}$}\sum_{v=1}^{2}\ _ih^{\tau_1\tau_2}\bigg(\frac{\partial\ _i\Gamma_{\tau_2k_v}^k }{\partial x^{\tau_1}}+\ _i\Gamma_{\tau_2k_v}^{\tilde{k}}\ _i\Gamma_{\tau_1\tilde{k}}^k\bigg)\otimes\ _ig_{k_1,..,k_{v-1},k,k_{v+1},...,k_p}^{\tau_1}\notag\\&-\mbox{$\frac{1}{2}$}\sum_{v=1}^{2}\sum_{w=1}^{2}\ _ih^{\tau_1\tau_2}\ _i\Gamma_{\tau_1k_v}^k\ _i\Gamma_{\tau_2k_w}^{\tilde{k}}\ \otimes\ _ig_{k_1,..,k_{v-1},k,k_{v+1},...,k_{w-1},\tilde{k},k_{w+1},...,k_p}^{\tau_1} \bigg]\biggr\}\times\notag\\& de dc d\xi+o\left(\epsilon^{\frac{1}{2}}\right),
\end{align}
where $\Delta x^{\tau_1}$ is time evolution of $x$ in the $\tau_1$ coordinate system, $\partial_{\tau_1} \Psi_s^{\tau,i}(x)=\frac{\partial \Psi_s^{\tau,i}(x)}{\partial \chi^{\tau_1}}$ and $\partial_{\tau_1}\partial_{\tau_2}\Psi_s^{\tau,i}(x)=\frac{\partial^2 \Psi_s^{\tau,i}(x)}{\partial \chi^{\tau_1}\partial \chi^{\tau_2}}$.  As we assume the strategy space is a metric affine space then, $\Delta x^{\lambda}=\Delta\xi^{\lambda}+\frac{1}{2}\Gamma_{\tau_1\tau_2}^\lambda\ \Delta\xi^{\tau_1}\Delta\xi^{\tau_2}-...$, where $\tau_1$ follows some spherical coordinate system, $\Delta\xi^{\tau_1}\equiv e_j^{\tau_1}\ \Delta q^j$ with $e_j^{\tau_1}(x)=\frac{\partial q^{\tau_1}}{\partial x^j}$ and $\Delta q$ is the change in $q$ in Cartesian coordinate system equivalent to any change of $x$ in polar coordinate system \cite{kleinert2009}. Here all affine connections are evaluated at the post point $x$ \cite{kleinert2009}.

\medskip

After using the value of $\Delta x^\lambda$ in the Equation (\ref{9.14}) we get, 
\begin{align}\label{9.15}
&\Psi_{s}^{\tau,i}(x)+\epsilon \frac{\partial \Psi_{s}^{\tau,i}(x)}{\partial s}+o(\epsilon)\notag\\& =\frac{1}{L_\epsilon} \frac{\sqrt{h}}{2}\int_{ \widetilde{\Sigma}_s} \left[\Psi_s^{\tau,i}(x)+\left(\Delta\xi^{\tau_1}+\mbox{$\frac{1}{2}$}\Gamma_{\tau_1\lambda}^{\tau_1} \Delta\xi^{\tau_2}\Delta\xi^{\lambda}\right) \partial_{\tau_1} \Psi_s^{\tau,i}(x)  + \mbox{$\frac{1}{2}$} \Delta \xi^{\tau_1} \Delta \xi^{\tau_2} \partial_{\tau_1}\partial_{\tau_2}\Psi_s^{\tau,i}(x)+ o\left(\epsilon^{\frac{1}{2}}\right)\right]\notag\\& \exp\biggr\{\iota\epsilon  \bigg[3+ h^{\tau_1\tau_2} \widehat{N}_{\tau_1\tau_2} (\pi_i b_i)^\Omega-\mbox{${\frac{1}{3!\sqrt h}}$} \varepsilon^{\hat{\tau}_1\hat{\tau}_2\hat{\tau}_3} \widehat{H}_{\hat{\tau}_1\hat{\tau}_2\hat{\tau}_3} (\pi_i b_i)^{1-\Omega}\notag\\&+\frac{e^{\tau_1\tau_2} \nabla^{\tau_1}c^{\tau_2}}{\pi\epsilon}-Q\widetilde{R}\bar {x}-\ _ig_{k_1,..,k_{v-1},k,k_{v+1},...,k_p}^{\tau_1}-\  _ig_{k_1,..,k_{v-1},k,k_{v+1},...,k_p;\ \nu}^{\tau_1}\notag\\&-\mbox{$\frac{1}{2}$}\sum_{v=1}^{2}\ _ih^{\tau_1\tau_2}\bigg(\frac{\partial\ _i\Gamma_{\tau_2k_v}^k }{\partial x^{\tau_1}}+\ _i\Gamma_{\tau_2k_v}^{\tilde{k}}\ _i\Gamma_{\tau_1\tilde{k}}^k\bigg)\otimes\ _ig_{k_1,..,k_{v-1},k,k_{v+1},...,k_p}^{\tau_1}\notag\\&-\mbox{$\frac{1}{2}$}\sum_{v=1}^{2}\sum_{w=1}^{2}\ _ih^{\tau_1\tau_2}\ _i\Gamma_{\tau_1k_v}^k\ _i\Gamma_{\tau_2k_w}^{\tilde{k}}\ \otimes\ _ig_{k_1,..,k_{v-1},k,k_{v+1},...,k_{w-1},\tilde{k},k_{w+1},...,k_p}^{\tau_1} \bigg]\biggr\} de dc d\xi+o\left(\epsilon^{\frac{1}{2}}\right).
\end{align}

\medskip

For a positive finite number $M\ra \infty$  the integral kernel of Equation (\ref{9.15}) is 
\begin{align}\label{9.16}
& K^\epsilon(x,\Delta x)\notag\\&=\mbox{$\frac{\sqrt h}{2\sqrt{2\pi\epsilon/M}}$} \exp\biggr\{\mbox{$\frac{\iota M}{2}$}  \bigg[3+ h^{\tau_1\tau_2} \widehat{N}_{\tau_1\tau_2} (\pi_i\ b_i)^\Omega-\mbox{${\frac{1}{3!\sqrt h}}$} \varepsilon^{\hat{\tau}_1\hat{\tau}_2\hat{\tau}_3} \widehat{H}_{\hat{\tau}_1\hat{\tau}_2\hat{\tau}_3} (\pi_i b_i)^{1-\Omega}\notag\\&+\frac{e^{\tau_1\tau_2} \nabla^{\tau_1}c^{\tau_2}}{\pi\epsilon}-Q\widetilde{R}\bar {x}-\ _ig_{k_1,..,k_{v-1},k,k_{v+1},...,k_p}^{\tau_1}-\  _ig_{k_1,..,k_{v-1},k,k_{v+1},...,k_p;\ \nu}^{\tau_1}\notag\\&-\mbox{$\frac{1}{2}$}\sum_{v=1}^{2}\ _ih^{\tau_1\tau_2}\bigg(\frac{\partial\ _i\Gamma_{\tau_2k_v}^k }{\partial x^{\tau_1}}+\ _i\Gamma_{\tau_2k_v}^{\tilde{k}}\ _i\Gamma_{\tau_1\tilde{k}}^k\bigg)\otimes\ _ig_{k_1,..,k_{v-1},k,k_{v+1},...,k_p}^{\tau_1}\notag\\&-\mbox{$\frac{1}{2}$}\sum_{v=1}^{2}\sum_{w=1}^{2}\ _ih^{\tau_1\tau_2}\ _i\Gamma_{\tau_1k_v}^k\ _i\Gamma_{\tau_2k_w}^{\tilde{k}}\ \otimes\ _ig_{k_1,..,k_{v-1},k,k_{v+1},...,k_{w-1},\tilde{k},k_{w+1},...,k_p}^{\tau_1} \bigg]\biggr\}.
\end{align}
Define the potential of the strategy field for firm $i$ is 
\begin{align}
V_i&=Q\widetilde{R}\bar {x}-\ _ig_{k_1,..,k_{v-1},k,k_{v+1},...,k_p}^{\tau_1}-\  _ig_{k_1,..,k_{v-1},k,k_{v+1},...,k_p;\ \nu}^{\tau_1}\notag\\&-\mbox{$\frac{1}{2}$}\sum_{v=1}^{2}\ _ih^{\tau_1\tau_2}\bigg(\frac{\partial\ _i\Gamma_{\tau_2k_v}^k }{\partial x^{\tau_1}}+\ _i\Gamma_{\tau_2k_v}^{\tilde{k}}\ _i\Gamma_{\tau_1\tilde{k}}^k\bigg)\otimes\ _ig_{k_1,..,k_{v-1},k,k_{v+1},...,k_p}^{\tau_1}\notag\\&-\mbox{$\frac{1}{2}$}\sum_{v=1}^{2}\sum_{w=1}^{2}\ _ih^{\tau_1\tau_2}\ _i\Gamma_{\tau_1k_v}^k\ _i\Gamma_{\tau_2k_w}^{\tilde{k}}\ \otimes\ _ig_{k_1,..,k_{v-1},k,k_{v+1},...,k_{w-1},\tilde{k},k_{w+1},...,k_p}^{\tau_1}\notag
\end{align}
such that,
\begin{align}\label{9.17}
& K^\epsilon(x,\Delta x)\notag\\&=\mbox{$\frac{\sqrt h}{2\sqrt{2\pi\epsilon/M}}$} \exp\biggr\{\mbox{$\frac{\iota M}{2}$}  \bigg[3+ h^{\tau_1\tau_2} \widehat{N}_{\tau_1\tau_2} (\pi_i b_i)^\Omega-\mbox{${\frac{1}{3!\sqrt h}}$} \varepsilon^{\hat{\tau}_1\hat{\tau}_2\hat{\tau}_3} \widehat{H}_{\hat{\tau}_1\hat{\tau}_2\hat{\tau}_3} (\pi_i b_i)^{1-\Omega}+\frac{e^{\tau_1\tau_2} \nabla^{\tau_1}c^{\tau_2}}{\pi\epsilon}-V_i \bigg]\biggr\}.
\end{align}
Hence, using a Laplace approximation, Equation (\ref{9.16}) becomes,
\begin{align}\label{9.18}
&\Psi_{s}^{\tau,i}(x)+\epsilon \frac{\partial \Psi_{s}^{\tau,i}(x)}{\partial s}+o(\epsilon)\notag\\& =\frac{1}{L_\epsilon}  \int_{ \widetilde{\Sigma}_s} K_0^\epsilon(x,\Delta \xi) \biggr[\Psi_s^{\tau,i}(x)+\left(\Delta\xi^{\tau_1}+\mbox{$\frac{1}{2}$}\ _i\Gamma_{\tau_2\lambda}^{\tau_1}\ \Delta\xi^{\tau_2}\Delta\xi^{\lambda}\right) \partial_{\tau_1} \Psi_s^{\tau,i}(x)\notag\\&+ \mbox{$\frac{1}{2}$} \Delta \xi^{\tau_1} \Delta \xi^{\tau_2} \partial_{\tau_1}\partial_{\tau_2}\Psi_s^{\tau,i}(x)+ o\left(\epsilon^{\frac{1}{2}}\right)\biggr],
\end{align}
with the zeroth-order kernel
\begin{align}\label{9.19}
 K_0^\epsilon(x,\Delta \xi)&=\mbox{$\frac{\sqrt h}{2L_\epsilon\sqrt{\frac{2\pi\iota\epsilon}{M\big|\hat{h}_{\tau_1\tau_2}\partial_{\tau_1}\partial_{\tau_2}F_0^i\big|}}}$} \exp\biggr\{\mbox{$\frac{\iota M}{2}$}  \bigg[3+ h^{\tau_1\tau_2} \widehat{N}_{\tau_1\tau_2} (\pi_i b_i)^\Omega\notag\\&-\mbox{${\frac{1}{3!\sqrt h}}$} \varepsilon^{\hat{\tau}_1\hat{\tau}_2\hat{\tau}_3} \widehat{H}_{\hat{\tau}_1\hat{\tau}_2\hat{\tau}_3} (\pi_i b_i)^{1-\Omega}+\frac{e^{\tau_1\tau_2} \nabla^{\tau_1}c^{\tau_2}}{\pi\epsilon}-V_i \bigg] \Delta\xi^{\tau_1}\Delta\xi^{\tau_2}\biggr\},
\end{align}
of unit normalization $$\int_{ \widetilde{\Sigma}_s}\ K_0^\epsilon(x,\Delta\xi) de dc d\xi=1,$$ as $M\ra\infty$, where $\partial_{\tau_1}\partial_{\tau_2}F_0^i$ is $i^{th}$ firm's Hessian of a strictly monotonically increasing function $F^i$ with the initial value $F_0^i$. Hence 
\begin{align}
F^i\ \hat{h}_{\tau_1\tau_2}&= 3+h^{\tau_1\tau_2} \widehat{N}_{\tau_1\tau_2} (\pi_i b_i)^\Omega-\mbox{${\frac{1}{3!\sqrt h}}$} \varepsilon^{\hat{\tau}_1\hat{\tau}_2\hat{\tau}_3} \widehat{H}_{\hat{\tau}_1\hat{\tau}_2\hat{\tau}_3} (\pi_i b_i)^{1-\Omega}+\frac{e^{\tau_1\tau_2} \nabla^{\tau_1}c^{\tau_2}}{\pi\epsilon}-V_i, 
\end{align}
and $\hat{h}^{\tau_1\tau_2}$ is a  contravariant metric of the $11$-dimensional background strategy super-field. After assuming $L_\epsilon=\sqrt{h\hat{h}}/2>0$, where $\hat h$ is the determinant of the background super-field, Equation (\ref{9.19}) becomes,
\begin{align}\label{9.190}
K_0^\epsilon(x,\Delta \xi)&=\mbox{$\frac{\sqrt{\hat h} }{\sqrt{\frac{2\pi\iota\epsilon}{M\ \big|\hat{h}_{\tau_1\tau_2}\partial_{\tau_1}\partial_{\tau_2}F_0^i\big|}}}$} \exp\biggr\{\mbox{$\frac{-\iota M}{2}$} F_0^i \hat{h}_{\tau_1\tau_2} \biggr\} \exp\biggr\{\mbox{$\frac{\iota M}{2}$} F^i\hat{h}_{\tau_1\tau_2} \biggr\}
\end{align}

Using the kernel in Equation (\ref{9.190}) we get a two-point correlation function
\begin{align}\label{9.22}
\langle\Delta\xi^{\tau_1} \Delta\xi^{\tau_2}\rangle=\int_{ \widetilde{\Sigma}_s} \Delta\xi^{\tau_1} \Delta\xi^{\tau_2} K_0^\epsilon(x,\Delta\xi) de dc d\xi=\frac{\iota\epsilon}{M} F_0^i  \hat{h}^{\tau_1\tau_2}.
\end{align}
Equations (\ref{9.22}) and (\ref{9.18}) yield,
\begin{align}\label{9.23}
&\Psi_{s}^{\tau,i}(x)+\epsilon \frac{\partial \Psi_{s}^{\tau,i}(x)}{\partial s}+o(\epsilon)\notag\\& =\Psi_s^{\tau,i}(x)+\mbox{$\frac{\iota\epsilon}{2M}$} F_0^i \left[ \hat{h}^{\tau_1\tau_2} \partial_{\tau_1} \partial_{\tau_2}\Psi_s^{\tau,i}(x)-\ _i\Gamma_{\tau_2}^{\tau_2\tau_1}\partial_{\tau_1}\Psi_s^{\tau,i}(x)\right]+ o\left(\epsilon\right).
\end{align}
The parenthesis term in Equation (\ref{9.23}) is proportional to covariant Laplacian of $i^{th}$ firm's $\Psi_s^i(x)$ \cite{kleinert2009}:
\begin{align}
D_{\tau_1}D^{\tau_1}\Psi_s^{\tau,i}\equiv\hat{h}^{\tau_1\tau_2}D_{\tau_1}D_{\tau_2}\Psi_s^{\tau,i}=\hat{h}^{\tau_1\tau_2}D_{\tau_1}\partial_{\tau_2}\Psi_s^{\tau,i}=\hat{h}^{\tau_1\tau_2} \partial_{\tau_1} \partial_{\tau_2}\Psi_s^{\tau,i}-\ _i\Gamma_{\tau_2}^{\tau_2\tau_1}\partial_{\tau_1}\Psi_s^{\tau,i}.\notag
\end{align}
Therefore, Shr\"odinger-like Equation in this case is
\begin{align}\label{9.24}
\partial_s \Psi_{s}^{\tau,i}(x)&=\mbox{$\frac{\iota}{2M}$} F_0^i \left[ \hat{h}^{\tau_1\tau_2} \partial_{\tau_1} \partial_{\tau_2}\Psi_s^{\tau,i}(x)-\ _i\Gamma_{\tau_2}^{\tau_2\tau_1}\partial_{\tau_1}\Psi_s^{\tau,i}(x)\right]=\mbox{$\frac{\iota}{2M}$}\ F_0^i\ D_{\tau_1}D^{\tau_1}\Psi_s^{\tau,i},
\end{align}
where $\partial_s \Psi_{s}^{\tau,i}(x)=\partial \Psi_{s}^{\tau,i}(x)/\partial s$. After differentiating Equation (\ref{9.24}) partially with respect to $\hat{\rho}$ and solving for it gives us optimal degree of cooperation at time $s$.
\end{proof}
 
\section{Discussion}
In this paper we show if a firm has a positive measure of market share then it creates a curvature around its strategy in the strategy spacetime which motivates it further to create more curvature and makes the new firms fall into it. Therefore, rationality of the firm is the curvature of the strategy spacetime. Secondly, we show if a firm sells the right of its product multiple times to the same consumer then its profit reduction operator increased and the firm shuts down. Finally, we show if the conformal, curved strategy space has 2-brane action function then we are able to find out the optimal level of semi-cooperation of that firm. In order to do so we assume, the property of strategy spacetime is similar to a $11$-dimensional Clifford torus and we are able to compactify the transverse directions to make it as a $3$-dimensional space. Throughout this paper we assume the background manifold is homeomorphic and diffeomorphic but in future we relax the assumption of diffeomorphism to discuss about optimal cooperation under exotic $7$-sphere.
\bibliographystyle{abbrv}
\bibliography{bib}
\end{document}